\documentclass[12pt]{amsart} \textwidth=14.5cm \oddsidemargin=1cm
\usepackage[rgb]{xcolor}
\usepackage{tikz}
\usepackage{verbatim}
\usepackage{amsmath}
\usepackage{amsxtra}
\usepackage{amscd}
\usepackage{amsthm}
\usepackage{amsfonts}
\usepackage{amssymb}
\usepackage{eucal}
\usetikzlibrary{arrows}

\usepackage{latexsym,todonotes}

\usepackage[linktocpage=true]{hyperref}
\hypersetup{colorlinks,linkcolor=blue,urlcolor=orange,citecolor=blue}
\usepackage[all, cmtip]{xy}

\usepackage{stmaryrd}
\usepackage{color} 
\newtheorem{thm}{Theorem} [section]
\newtheorem{cor}[thm]{Corollary}
\newtheorem{lem}[thm]{Lemma}
\newtheorem{prop}[thm]{Proposition}

\theoremstyle{definition}
\newtheorem{definition}[thm]{Definition}

\theoremstyle{remark}
\newtheorem{rem}[thm]{Remark}

\numberwithin{equation}{section}
\begin{document}

%Referring commands:
\newcommand{\thmref}[1]{Theorem~\ref{#1}}
\newcommand{\secref}[1]{Section~\ref{#1}}
\newcommand{\lemref}[1]{Lemma~\ref{#1}}
\newcommand{\propref}[1]{Proposition~\ref{#1}}
\newcommand{\corref}[1]{Corollary~\ref{#1}}
\newcommand{\remref}[1]{Remark~\ref{#1}}
\newcommand{\eqnref}[1]{(\ref{#1})}

\newcommand{\exref}[1]{Example~\ref{#1}}

%Theorem for the introduction
\newtheorem{innercustomthm}{{\bf Theorem}}
\newenvironment{customthm}[1]
  {\renewcommand\theinnercustomthm{#1}\innercustomthm}
  {\endinnercustomthm}
  
  \newtheorem{innercustomcor}{{\bf Corollary}}
\newenvironment{customcor}[1]
  {\renewcommand\theinnercustomcor{#1}\innercustomcor}
  {\endinnercustomthm}
  
  \newtheorem{innercustomprop}{{\bf Proposition}}
\newenvironment{customprop}[1]
  {\renewcommand\theinnercustomprop{#1}\innercustomprop}
  {\endinnercustomthm}

% Simplified symbols:
\newcommand{\B}{{\mathbf B}}
\newcommand{\ch}{\text{ch }}
\newcommand{\C}{{\mathbb C}}
\newcommand{\N}{{\mathbb N}}
\newcommand{\F}{{\mf F}}
\newcommand{\Q}{\mathbb {Q}}
\newcommand{\Z}{{\mathbb Z}}
\newcommand{\K}{{\mathbb Q}}

\newcommand{\la}{\lambda}
\newcommand{\ep}{\epsilon}
\newcommand{\h}{\mathfrak h}
\newcommand{\n}{\mf n}
\newcommand{\A}{{\mathcal A}}
\newcommand{\aA}{{}_\A}
\newcommand{\aAp}{{}_\A'}
\newcommand{\aL}{{}_\A L}
\newcommand{\aM}{{}_\A M}
\newcommand{\Hom}{\text{Hom}\,}
\newcommand{\Ind}{\text{Ind}\,}
\newcommand{\D}{\mc D}
\newcommand{\La}{\Lambda}
\newcommand{\dt}{\mathord{\hbox{${\frac{d}{d t}}$}}}
\newcommand{\E}{\EE}
\newcommand{\ba}{\tilde{\pa}}
\newcommand{\half}{\frac{1}{2}}
\newcommand{\llb}{\llbracket}
\newcommand{\mc}{\mathcal}
\newcommand{\mf}{\mathfrak}
\newcommand{\hf}{\frac{1}{2}}
\newcommand{\id}{\text{id}}
\newcommand{\hgl}{\widehat{\mathfrak{gl}}}
\newcommand{\gl}{{\mathfrak{gl}}}
\newcommand{\hz}{\hf+\Z}
\newcommand{\I}{{I}}
\newcommand{\dinfty}{{\infty\vert\infty}}
\newcommand{\SP}{\overline{\mathcal P}}
\newcommand{\ov}{\overline}
\newcommand{\one}{\bold{1}}
\newcommand{\rrb}{\rrbracket}
\newcommand{\wt}{\widetilde}
\newcommand{\wh}{\widehat}
\newcommand{\sL}{\ov{\mf{l}}}
\newcommand{\hh}{\widehat{\mf{h}}}
\newcommand{\even}{{\bar 0}}
\newcommand{\odd}{{\bar 1}}
\newcommand{\mscr}{\mathscr}
\newcommand{\hZ}{{\underline{\mathbb Z}}}
\newcommand{\vs}{\varsigma}

\newcommand{\blue}[1]{{\color{blue}#1}}
\newcommand{\red}[1]{{\color{red}#1}}
\newcommand{\green}[1]{{\color{green}#1}}
\newcommand{\white}[1]{{\color{white}#1}}
\newcommand{\ang}[1]{\langle #1 \rangle}
\newcommand{\wtd}{\widetilde}
\newcommand{\cbinom}[2]{\left\{ \begin{matrix} #1\\#2 \end{matrix} \right\}}
\newcommand{\bbinom}[2]{\begin{bmatrix}#1 \\ #2\end{bmatrix}}
\newcommand{\qbinom}[2]{\begin{bmatrix}
		#1\\#2
\end{bmatrix} }

\newcommand{\aaf}{\mathfrak a}
\newcommand{\bb}{\mathfrak b}
\newcommand{\cc}{\mathfrak c} 
\newcommand{\g}{\mathfrak g}
\newcommand{\qq}{\mathfrak q}
\newcommand{\qn}{\mathfrak q (n)}
\newcommand{\VV}{\mathbb V}
\newcommand{\WW}{\mathbb W}
\newcommand{\TT}{\mathbb T}
\newcommand{\EE}{\mathbb E}
\newcommand{\FF}{\mathbb F}
\newcommand{\KK}{\mathbb K}
\newcommand{\uu}{\mathfrak u}
\newcommand{\SU}{\overline{\mathfrak u}}

%% quantum covering groups
\newcommand{\ff}{\bf{f}}
\newcommand{\ffpr}{{}'\bf{f}}
\newcommand{\UU}{\mathbf U}
\newcommand{\UUi}{\UU^\imath}
\newcommand{\ipsi}{\psi_{\imath}}
\newcommand{\UUdot}{\dot{\bold{U}}}
\newcommand{\UUidot}{\dot{\UU}^{\imath}}
\newcommand{\tK}{\widetilde{K}}
\newcommand{\tJ}{\widetilde{J}}

\newcommand{\End}{\text{End}\,}

%% added by Ian
\newcommand{\restr}[1]{\hspace{-0.3em}\mid_{#1}}

\newcommand{\tya}[1]{{}^A\!#1}
\newcommand{\tyb}[1]{{}^B\!#1}
\newcommand{\tyd}[1]{{}^D\!#1}
\newcommand{\tyg}[1]{{}^G\!#1}
\newcommand{\tyeven}[1]{{#1}^+}
\newcommand{\tyodd}[1]{{#1}^-}
\newcommand{\aeven}[1]{\tya\tyeven{#1}}
\newcommand{\aodd}[1]{\tya\tyodd{#1}}
\newcommand{\beven}[1]{\tyb\tyeven{#1}}
\newcommand{\bodd}[1]{\tyb\tyodd{#1}}
\newcommand{\deven}[1]{\tyd\tyeven{#1}}
\newcommand{\dodd}[1]{\tyd\tyodd{#1}}
\newcommand{\geven}[1]{\tyg\tyeven{#1}}
\newcommand{\godd}[1]{\tyg\tyodd{#1}}

\newcommand{\pae}{\aeven\partial}
\newcommand{\pao}{\aodd\partial}
\newcommand{\pbe}{\beven\partial}
\newcommand{\pbo}{\bodd\partial}
\newcommand{\pde}{\deven\partial}
\newcommand{\pdo}{\dodd\partial}
\newcommand{\pge}{\geven\partial}
\newcommand{\pgo}{\godd\partial}

\newcommand{\san}{\tya{s}}
\newcommand{\sae}{\aeven{s}}
\newcommand{\sao}{\aodd{s}}
\newcommand{\sbn}{\tyb{s}}
\newcommand{\sbe}{\beven{s}}
\newcommand{\sbo}{\bodd{s}}
\newcommand{\sdn}{\tyd{s}}
\newcommand{\sde}{\deven{s}}
\newcommand{\sdo}{\dodd{s}}
\newcommand{\sgn}{\tyg{s}}
\newcommand{\sge}{\geven{s}}
\newcommand{\sgo}{\godd{s}}

\newcommand{\Res}{\text{Res}\,}

\usetikzlibrary{cd}

%new 
\newcommand{\al}{\alpha}
\newcommand{\bx}{\texttt{x}}
\newcommand{\Cl}{{\mathcal Cl}}
\newcommand{\HH}{\mathfrak H}
\newcommand{\hgt}{\text{ht}}
\newcommand{\nH}{\mathcal H^-}
\newcommand{\osp}{\mathfrak{\lowercase{osp}}}
\newcommand{\ds}{\displaystyle}
\newcommand{\wgt}{\text{wt}}
\newcommand{\spn}{\text{span}}
\newcommand{\diag}{\text{diag}}

\newcommand{\Iblack}{\I_{\bullet}}
\newcommand{\wb}{w_\bullet}
\newcommand{\UIblack}{\U_{\Iblack}}

% i{}divided powers
%\newcommand{\dvev}[1]{{\mathfrak{t}}_{\ev}^{{(#1)}}}
%\newcommand{\dvo}[1]{{\mathfrak{t}}_{\odd}^{{(#1)}}}
%\newcommand{\dvd}[1]{t_{\odd}^{{(#1)}}}
%\newcommand{\dvp}[1]{t_{\ev}^{{(#1)}}}

% canonical bases
\newcommand{\cB}{{\bf B}} % changed from mathcal to bold type\

\title[Canonical bases arising from $\imath$quantum covering groups]
{Canonical bases arising from $\imath$quantum \\
covering groups of Kac-Moody type}

\author[Christopher Chung]{Christopher Chung}
\address{Okinawa Institute of Science and Technology, Okinawa, Japan 904-0495}
\email{christopher.chung@oist.jp}

\begin{abstract}  
For $\imath$quantum covering groups $(\UU, \UU^\imath)$ of super Kac-Moody type, we construct $\imath$-canonical bases for the highest weight integrable $\UU$-modules and their tensor products regarded as $\UU^\imath$-modules, as well as a canonical basis for the modified form $ \UUidot $ of the $\imath$quantum group $\UU^\imath$, using the $ \imath^\pi $-divided powers, rank one canonical basis for $\UUi$.  
\end{abstract}

%\vspace{.3cm}
\maketitle

\setcounter{tocdepth}{1}
\tableofcontents

%%%%%
%%%%%
\section{Introduction}

\subsection{Background}
\label{subsec:history}

A quantum symmetric pair $(\UU,\UUi)$ is a quantization of the symmetric pair of enveloping algebras $ (\UU(\g), \UU(\g^\theta)) $ where $ \theta : \g \to \g $ is an involution of the Lie algebra $ \g $. Originally developed for applications in harmonic analysis for quantum group analogs of symmetric spaces, G.~Letzter developed a comprehensive theory of quantum symmetric pairs for all semisimple $\g$ in \cite{Le99}. The algebraic theory of quantum symmetric pairs was subsequently extended to the setting of Kac-Moody algebras in \cite{Ko14}. The \emph{$\imath$quantum group} $ \UUi $ is a subalgebra of the quantum group $ \UU $ satisfying a \emph{coideal property}; coideal subalgebras provide important substructure for $ \UU $, since Hopf subalgebras are rare `in nature'.

More recent developments have made it apparent that quantum symmetric pairs play an important role in representation theory at large. In a series of papers,  H.~Bao and W.~Wang proposed a program of canonical bases for quantum symmetric pairs \cite{BW18a,BW18b,BW18c}. They performed their program for the Type AIII/IV symmetric pairs $(\mathfrak{sl}_{2N},\mathfrak{s}(\gl_N\times \gl_N))$ and $(\mathfrak{sl}_{2N+1},\mathfrak{s}(\gl_N\times \gl_{N+1}))$ and applied it to tensor products of their $ \UUi $-modules, establishing a Kazhdan-Lusztig theory and irreducible character formula for the category $\mc{O}$ of the ortho-symplectic Lie superalgebra $\mathfrak{osp}(2n{+}1\,|\,2m)$. Together with previously known results, these recent developments suggest that quantum symmetric pairs allow as deep a theory as quantized enveloping algebras themselves. In fact, $\UU$ can be viewed as a special type of quantum symmetric pair, the \emph{diagonal quantum symmetric pair} ($\UU \otimes \UU$, $ \imath(\UU) $) where $ \imath = (\omega \otimes 1)\Delta : \UU \to \UU \otimes \UU $. It is thus reasonable to expect that many results about quantized groups have their counterparts in the realm of quantum symmetric pairs. 

%In recent years, it has become increasingly clear that a number of fundamental constructions for quantum groups admit highly nontrivial generalizations to the setting of QSP. In \cite{BW18b} (also cf. \cite{BW18a}), the authors developed a theory of $\imath$-canonical bases for the QSP $(\UU, \UUi)$ {\em of finite type}, on $\UUi$-modules as well as a modified form $\UUidot$. The constructions of quasi K-matrix and universal K-matrix for QSP \cite{BW18a, BK18} have played a significant role too. 
%The $\imath$quantum group $\UUi$ of quasi-split type AIII and its $\imath$-canonical basis have applications to super Kazhdan-Lusztig theory of type BCD \cite{BW18a, Bao17}, admit a geometric realization \cite{BKLW, LW18} and a KLR type categorification \cite{BSWW}. 
%%As demanded by the development of $\imath$canonical bases, the notions of quasi K-matrix and K-matrix for QSP were formulated by the authors \cite{BW18a} and developed more fully by Balagovic-Kolb in \cite{BK18}.
%

A quantum covering group $\UU_\pi$, introduced in 
%Series of papers by Clark, Hill, and Wang in
\cite{CHW13} is an algebra defined via a super Cartan datum $I$ (a finite indexing set associated to Kac-Moody superalgebras with no isotropic odd roots). $ \UU_\pi $ depends on two parameters $ q $ and $ \pi $, where $ \pi^2 = 1 $. A quantum covering group specializes at $ \pi = 1 $ to the quantum group above, and at $ \pi = -1 $ to a quantum supergroup of anisotropic type (see \cite{BKM98}). In addition to the usual Chevalley generators, we have generators $J_i$ for each $ i \in I $. If one writes $ K_i $ as $ q^{h_i} $, then analogously we will have $ J_i = \pi^{h_i} $. 
% A crucial advantage of these new generators: integrable modules for the corresponding quantum supergroups for all dominant integral weights. 
The parameter $\pi$ can be seen as a shadow of a parity shift functor in e.g. D.~Hill and W.~Wang's (\cite{HW15}) categorification of quantum groups by the \emph{spin} quiver Hecke superalgebras introduced in \cite{KKT16}. Since then, further progress has been made on the odd/spin/super categorification of quantum covering groups; see \cite{KKO14, EL16, BE17}. 

Much of the theory for quantum groups, have parallel constructions in the realm of quantum covering groups. In particular, a theory of canonical bases for integrable modules of $ \UU_\pi $ and its modified (idempotented) form $ \UUdot_\pi $ has been developed, in \cite{CHW14,Cl14}. 

\subsection{$\imath^\pi$-divided powers}
For the negative half $ \UU^- $ of the quantum group in rank one $ \UU = \UU_q(\mf{sl}_2) $, the Lusztig divided powers are monomials in a single variable $F$, and they form the canonical basis for $ \UU^- $. The canonical basis for $ \UUi $ in rank one is formed by the \emph{$\imath$-divided powers}, introduced in \cite{BW18b, BW18c} and further explored in \cite{BeW18}. Instead of being monomials, they are polynomials in a single variable $ B $. They give bases for finite-dimensional simple $\mf{sl}_2$-modules, and have two different formulas, $ B^{(n)}_{\even} $ and $B^{(n)}_{\odd}$, depending on the parity of the corresponding highest weight, which is a non-negative integer. 
The $\imath$-divided powers and their expansion formulas in \cite{BeW18} formed a cornerstone of the construction of the Serre presentation for quasi-split $\imath$-quantum groups established in H.~Chen, M.~Lu and W.~Wang in \cite{CLW18}. In \cite{BW18b,BW18c}, $\imath$-divided powers for $ i \in I $ with $ \tau i = i $ were defined using the same formulas, and then shown to generate as an algebra the integral form $ {}_\A \UUidot $ of the modified quantum group. In \cite{C19}, the $\imath$-divided powers above are shown to have a generalization to $ \UUi_\pi $, the \emph{$\imath^\pi$-divided powers} $ B_{i,\odd}^{(m)} $ and $ B_{i,\even}^{(m)} $ which are given in the formulas \eqref{eq:piDPodd} and \eqref{eq:piDPev} below for $ i \in I $ with $ \tau i = i $. The new facets $\pi$ and $J$ of quantum covering groups are incorporated into these formulas, and when we specialize at $ \pi = 1 $ and $ \tJ_i = 1 $, we obtain the $\imath$-divided powers above. The $ \imath^\pi $-divided powers also satisfy a collection of expansion formulas which are used to give a Serre presentation for $ \UUi_\pi $ and define a bar-involution on $ \UUi $. 

\subsection{Quasi $K$-matrix and canonical basis for $\UUi_\pi$}
For regular quantum groups, the bar involutions $ \ipsi $ on $ \UUi $ and $ \psi $ on $ \UU $ are not compatible; $ \ipsi $ is not simply the restriction of $ \psi $ to the subalgebra $ \UUi $. However, one can define a quasi-$K$-matrix $ \Upsilon $ that `intertwines' these two bar involutions. In the case of the diagonal quantum symmetric pair, the quasi-$K$-matrix arises naturally from Lusztig's quasi $\mc R$-matrix. The quasi-$K$-matrix is applied in \cite{BW18b,BW18c} to transform involutive based $ \UU $-modules ($\UU$-modules with distinguished bases compatible with the bar-involution $\psi$ on $\UU$), into involutive based $ \UUi $-modules, compatible with the bar-involution $\ipsi$ on $\UUi$.

The quasi-$K$-matrix $\Upsilon$ is invertible, and its inverse is obtained by applying the bar involution. Crucially, $ \Upsilon $ has the property that it preserves the integrality of the $ \A $-forms of integrable highest weight $ \UUi_\pi $-modules and their tensor products. Using this property of integrality of the action of their quasi-$K$-matrix, Bao and Wang defined in {\em loc. cit.} a new bar involution on based $ \UU $-modules (modules $M$ with a distinguished basis $B$, and compatible involution $ \psi $) %including integrable highest weight $\UU$-modules and their tensor products, 
thus enabling the construction of $\imath$-canonical bases of these modules (which are now based $ \UUi $-modules) from their canonical bases. With the $ \imath^\pi $-divided powers above, these constructions also lead to a theory of canonical basis for integrable based $ \UUi_\pi $-modules, which we develop here in this article. 

\subsection{Organization}
The rest of this article is organized as follows. In the next section, we introduce basic notation and notions for quantum covering groups. Then, in section~\ref{sec:Ui} we describe $ \UUi_\pi $ and the $ \imath^\pi $-divided powers. In section~\ref{sec:quasiKmatrix}, the quasi-$K$-matrix $\Upsilon$ for $\UUi_\pi$ is constructed and in section~\ref{sec:upintegrality} the integrality of its action is established, by which we mean that $\Upsilon$ preserves the integral $\A$-forms on integrable highest weight $\UU_\pi$-modules and their tensor products. We conclude by constructing the $\imath$-canonical basis for based $\UUi_\pi$-modules in a section~\ref{sec:iCBmodules} followed by canonical basis for the modified form $ \UUidot_\pi $ in section~\ref{sec:CBmodified} generalizing \cite{BW18b,BW18c}. \\

\noindent
{\bf Remark on notation.}
For the remaining sections we will drop the subscript $ \pi $ from $ \UU_\pi $ and related notation in the following chapters, so $ \UU $ will be understood to refer to the quantum covering group going forward. We will explicitly mention when we are referring to the usual quantum group e.g. when we specialize $ \pi = 1 $. 

\subsection*{Acknowledgments.}
The author is indebted to advisor Weiqiang Wang for his guidance and patience. This research is partially supported by Wang's NSF grant DMS-1702254, including GRA supports for the author.

%%%%%
%%%%%
\section{Quantum covering groups and canonical bases}
\label{sec:QCG}

In this section, we will recall the definition of a quantum covering group from \cite{CHW13} starting with a {\em super Cartan datum} and a root datum. A {\em Cartan datum} is a pair $(I,\cdot)$ consisting of a finite
set $I$ and a symmetric bilinear form $\nu,\nu'\mapsto \nu\cdot\nu'$
on the free abelian group $\Z[I]$ with values in $\Z$ satisfying
\begin{enumerate}
	\item[(a)] $d_i=\frac{i\cdot i}{2}\in \Z_{>0}$;
	
	\item[(b)]
	$2\frac{i\cdot j}{i\cdot i}\in -\N$ for $i\neq j$ in $I$, where $\N
	=\{0,1,2,\ldots\}$.
\end{enumerate}
If the datum can be decomposed as $ I=I_{\bf \even} \coprod I_{\bf \odd} $ such that
\begin{enumerate}
	\item[(c)] $I_{\bf \odd}\neq\emptyset$,
	\item[(d)] $2\frac{i\cdot j}{i\cdot i} \in 2\Z$ if $i\in I_{\bf \odd}$,
	\item[(e)] $d_i\equiv p(i) \mod 2, \quad \forall i\in I.$
\end{enumerate}
then we will called it a (bar-consistent) {\em super Cartan datum}. Condition [(e)] is known as the `bar-consistency' condition and is almost always satisfied for super Cartan data of finite or affine type (with one exception).

Note that (d) and (e) imply that
\begin{enumerate}
	\item[(f)] $i\cdot j\in 2\Z$ for all $i,j\in I$.
\end{enumerate}

The $i\in I_{\bf \even}$ are called even, $i\in I_{\bf \odd}$ are called odd. We
define a parity function $p:I\rightarrow\{0,1\}$ so that $i\in
I_{\ov{p(i)}}$. We extend this function to the homomorphism
$p:\Z[I]\rightarrow \Z$. Then $p$ induces a $\Z_2$-grading on
$\Z[I]$ which we shall call the parity grading.
%We define the {\em height} of $\nu=\sum_{i\in I}\nu_i i\in \Z[I]$ by $\height(\nu)=\sum \nu_i$.

A super Cartan datum $(I,\cdot)$ is said to be of {\em finite}
(resp. {\em affine}) type exactly when $(I,\cdot)$ is of
finite (resp. affine) type as a Cartan datum (cf. \cite[\S~2.1.3]{Lu94}).
In particular, the only super Cartan datum of finite type is the one
corresponding to the Lie superalgebras of type
$B(0,n)$ for $n\geq 1$ i.e. the orthosymplectic Lie superalgebras $ \osp(1|2n) $.

A {\em root datum} associated to a super Cartan datum $(I,\cdot)$
consists of
\begin{enumerate}
	\item[(a)]
	two finitely generated free abelian groups $Y$, $X$ and a
	perfect bilinear pairing $\ang{\cdot, \cdot}:Y\times X\rightarrow \Z$;
	
	\item[(b)]
	an embedding $I\subset X$ ($i\mapsto i'$) and an embedding $I\subset
	Y$ ($i\mapsto i$) satisfying
	
	\item[(c)] $\ang{i,j'}=\frac{2 i\cdot j}{i\cdot i}$ for all $i,j\in I$.
\end{enumerate}
We will always assume that the root datum is {\em $X$-regular} (respectively {\em $Y$-regular}) image of the embedding $I\subset X$
(respectively, the image of the embedding $I\subset Y$) is linearly independent in $X$ (respectively, in $Y$).

We also define a partial order $\leq$ on the weight lattice $X$ as follows: for $\la, \la' \in X$, 
\begin{equation}
\label{eq:leq}
\lambda \le \lambda' \text{ if and only if } \lambda' -\lambda \in \N[\I]. 
\end{equation}

The matrix $ A := (a_{ij}) := \ang{i,j'} $ is a {\em symmetrizable generalized super Cartan matrix}: if $ D = \diag(d_i\,|\, i \in I) $, then $ DA $ is symmetric.

%Let $X^+=\{\lambda\in X\mid \ang{i,\lambda}\in \N \text{ for all } i
%	\in I\}$. Note that there are no additional ``evenness'' assumptions for $X^+$.

Let $\pi$ be a parameter such that
$$\pi^2=1.
$$
For any $i\in I$, we set
$$q_i=q^{i\cdot i/2}, \qquad \pi_i=\pi^{p(i)}.
$$
Note that when the datum is consistent, $\pi_i=\pi^{\frac{i\cdot
		i}{2}}$; by induction, we therefore have
$\pi^{p(\nu)}=\pi^{\nu \cdot \nu/2}$ for $\nu\in \Z[I]$. 
We extend this notation so that if
$\nu=\sum \nu_i i\in \Z[I]$, then
$$
q_\nu=\prod_i q_i^{\nu_i}, \qquad \pi_\nu=\prod_i \pi_i^{\nu_i}.
$$
For any ring $R$ we define a new ring $R^\pi =R[\pi]/(\pi^2-1)$ 
(with $\pi$ commuting with $R$). Below, we will work over $\K(q)^\pi$ where $\K$ is a field of characteristic $0$ and occasionally $\A^\pi$ where $ \A := \Z[q,q^{-1}] $.

Recall also the {\em $(q,\pi)$-integers} and {\em $(q,\pi)$-binomial coefficients} in \cite{CHW13}: we shall denote
\[
[n]=\bbinom{n}{1}=\frac{({\pi} q)^n-q^{-n}}{{\pi} q - q^{-1}}\quad
\text{for } n\in\Z,
\]
\[
[n]^!=\prod_{s=1}^n [s]\quad \text{for } n\in\N,
\]
and with this notation we have
\[
\bbinom{m}{n}=\frac{[m]^!}{[n]^![m-n]^!}\quad \text{for
}0\leq n \leq m.
\]
We denote by $[n]_{i}, [m]_{i}^!,$ and $\qbinom{n}{m}_{i}$ the variants of $[n], [m]!,$ and $\qbinom{n}{m}$ with $q$ replaced by $q_i$ and $ \pi $ replaced by $ \pi_i $, and $ \bbinom{m}{n}_{q^2} $ the variant with $q$ replacing $ q^2 $. 
%Note that the the bar-consistency condition above is required for the bar-invariance of the $(q,\pi)$-integers $[n]_i$ and the $(q,\pi)$-binomial
%coefficients.

For any $i\neq j$ in $I$, we define the following polynomial in two (noncommutative) variables $x$ and $y$:
\begin{equation}
\label{qpipoly}
F_{ij}(x,y) = \sum_{n=0}^{1-a_{ij}}(-1)^{n}\pi_i^{np(j)+\binom{n}{2}} \bbinom{1 - a_{ij}}{n}_i x^n y x^{1 - a_{ij} - n}.
\end{equation}
Also, we have 

Assume that a root datum $(Y,X, \ang{\,,\,})$
of type $(I,\cdot)$ is given. The {\em quantum covering group} $ \UU $ of
type $(I,\cdot)$ is the associative $\K(q)^\pi$-superalgebra with generators
\[\
E_i\quad(i\in I),\quad F_i\quad (i\in I), \quad J_{\mu}\quad (\mu\in
Y),\quad K_\mu\quad(\mu\in Y),
\]
with parity $p(E_i)=p(F_i)=p(i)$ and
$p(K_\mu)=p(J_\mu)=0$, subject to the relations (a)-(f) below for
all $i, j \in I, \mu, \mu'\in Y$:
\[
\tag{R1} K_0=1,\quad K_\mu K_{\mu'}=K_{\mu+\mu'},
\]
\[
\tag{R2} J_{2\mu}=1, \quad J_\mu J_{\mu'}=J_{\mu+\mu'},
\]
\[
\tag{R3} J_\mu K_{\mu'}=K_{\mu'}J_{\mu},
\]
\[
\tag{R4} K_\mu E_i=q^{\ang{\mu,i'}}E_iK_{\mu}, \quad
J_{\mu}E_i=\pi^{\ang{\mu,i'}} E_iJ_{\mu},
\]
\[
\tag{R5} \; K_\mu F_i=q^{-\ang{\mu,i'}}F_iK_{\mu}, \quad J_{\mu}F_i=
\pi^{-\ang{\mu,i'}} F_iJ_{\mu},
\]
\[
\tag{R6} E_iF_j-\pi^{p(i)p(j)}
F_jE_i=\delta_{i,j}\frac{\wtd{J}_{i}\wtd{K}_i-\wtd{K}_{-i}}{\pi_iq_i-q_i^{-1}},
\label{R6}
\]
\[
\tag{R7}(q,\pi)\text{-Serre relations}\qquad F_{ij}(E_i,E_j) = 0 = F_{ij}(F_i,F_j), \text{ for all }i \neq j.
\label{R7}
\]
where for any element $\nu=\sum_i \nu_i i\in \Z[I]$ we have set
$\wtd{K}_\nu=\prod_i K_{d_i\nu_i i}$, $\wtd{J}_\nu=\prod_i
J_{d_i\nu_i i}$. In particular, $\wtd{K}_i=K_{d_i i}$,
$\wtd{J}_i=J_{d_i i}$. Under the bar-consistency condition,
$\wtd{J}_i=1$ for $i\in I_{\bf \even}$ while $\wtd{J}_i=J_i$ for $i\in
I_{\bf \odd}$. Note that by the same condition $ a_{ij} $ is always even for $ i \in I_{\bf \odd} $, and so $ J_i $ is central for all $i \in I$. As usual, denote by $ \UU^- $, $ \UU^+ $ and $ \UU^0 $ the subalgebras of $ \UU $ generated by $ \{ E_i \,|\, i \in I \} $, $ \{ F_i \,|\, i \in I \} $ and $ \{ J_\mu, K_\mu \,|\, \mu \in Y \} $ respectively. Also denote $ \UU^{0'} = \{ J_i, K_i \,|\, i \in I \} $.

Note that the $(q,\pi)$-Serre relations \eqref{R7} can be rewritten as
\begin{equation}
\label{qpiFSerre}
\sum_{n=0}^{1-a_{ij}}(-1)^{n}\pi_i^{np(j)+\binom{n}{2}}  F_i^{(n)} F_j F_i^{(1 - a_{ij} - n)} = 0
\end{equation}
and
\begin{equation}
\label{qpiESerre}
\sum_{n=0}^{1-a_{ij}}(-1)^{n}\pi_i^{np(j)+\binom{n}{2}}  E_i^{(n)} E_j E_i^{(1 - a_{ij} - n)} = 0,
\end{equation}
where we write $ F_i^{(n)} = F_i^n/[n]^!_{i} $ and $ E_i^{(n)} = E_i^n/[n]^!_{i} $ for $ n \geq 1 $ and $ i \geq 1 $.

Define $\ffpr$ to be the free associative $\Q(q)^\pi$-superalgebra
with $1$ and with even generators $\theta_i$ for $i\in I_0$ and
odd generators $\theta_i$ for $i\in I_\odd$. We abuse notation and
define the parity grading on $\ffpr$ by $p(\theta_i)=p(i)$.
We also have a weight grading $|\cdot|$ on $\ffpr$ defined
by setting $|\theta_i|=i$. 

By \cite[Prop~1.4.1]{CHW13}, there exists a unique symmetric bilinear form $(\cdot,\cdot)$ on $\ffpr$ with values in $\Q$ such that $(1,1)=1$ and 
\begin{equation}
\label{eq:bilinij}
(\theta_i, \theta_j) = \delta_{ij} (1-\pi_i q_i^{-2})^{-1} \text{ for all } i,j\in I. 
\end{equation}

Let $\mathcal{I}$ to denote the radical of $(\cdot, \cdot)$ which is a 2-sided ideal of $\ffpr$, and let $\ff=\ffpr/\mathcal{I}$ be the quotient algebra of $\ffpr$ by its radical. There exists well-defined algebra homomorphisms $\ff\to\UU$: $ x \mapsto x^+ $ with $ \theta_i^+ = E_i $ and image $ \UU^+ $, and $ x \mapsto x^+ $ with $ \theta_i^- = F_i $ and image $ \UU^- $. The algebra $\ff$ has weight space decomposition $\ff=\bigoplus_{\nu} \ff_\nu$ where
$\ff_\nu$ is the image of $\ffpr_\nu$, the weight space of  $\ffpr$ with weight $ \nu = \sum \nu_i i \in \Z[\I] $. We will denote the height of $\nu$ by $\hgt (\mu) =\sum_{i\in \I} \nu_i$ and for any $x \in \ff_\nu$, we set $|x| = \nu$.  Each weight space is finite dimensional. The symmetric bilinear form on $ \ffpr $ descends to a symmetric bilinear form on $\ff$ which is non-degenerate on each weight space.

\subsection{\bf The twisted derivations $r_i$ and $_i r$}
\label{subsec:rtwisted}

Let $i\in I$. There exist unique $\Q(q)^\pi$-linear maps
$r_i,_i r:\ffpr\rightarrow\ffpr$ such that $r_i(1)=_i r(1)=0$ and
$r_i(\theta_j)=_i r(\theta_j)=\delta_{ij}$ satisfying
\[
_i r(xy)=_i r(x)y+\pi^{p(x)p(i)}q^{|x|\cdot i}x_i r(y)
\]
\[
r_i(xy)=\pi^{p(y)p(i)}q^{|y|\cdot i}r_i(x)y+xr_i(y)
\]
for homogeneous $x,y\in\ffpr$. We see that if $x\in \ffpr_\nu$, then
$_i r(x),r_i(x)\in\ffpr_{\nu-i}$ and moreover,
\begin{equation}  \label{eq:derivadjunct}
	(\theta_i y, x) = (\theta_i,\theta_i)(y, _i r(x)), \quad
	(y \theta_i, x) = (\theta_i,\theta_i)(y, r_i(x))
\end{equation}
for all $x,y\in \ffpr$, and both maps descend to
maps on $\ff$ cf. \cite[\S 1.5]{CHW13}. 

The following lemmas on the twisted derivation will be important tools for the construction of the quasi K-matrix in part III. The first is \cite[Lemma~1.5.2]{CHW13}, a direct generalization of \cite[Lemma~1.2.15]{Lu94} for quantum groups: 
\begin{lem}\label{lem:rivanishing}
	Let $x\in \ff_\nu$ where $\nu\in \N[I]$ is nonzero.
	\begin{enumerate}
		\item[(a)] If $r_i(x)=0$ for all $i\in I$, then $x=0$.
		\item[(b)] If ${}_i r(x)=0$ for all $i\in I$, then $x=0$.
	\end{enumerate}
\end{lem} 

The following lemma is a generalization of \cite[Lemma~1.1]{BW18a} and will play a similar role in our setting:
\begin{lem} \label{lem:rijr}
	$ {}_j r \circ r_i = r_i \circ {}_j r $ for all $ i,j \in I $
\end{lem}

\begin{proof}
	It suffices to show this for homogeneous $ x \in \ffpr_{\mu} $, using induction on the height of $\mu$; for $ x = 1 $ both sides are identically $0$, and from their definition, we have 
	\begin{align*}
	r_j \circ {}_i r (xy) = {}_i r(x) r_j(y) + \pi^{p(y) p(j)} q^{|y| \cdot j} & r_j({}_i r(x)) y + \pi^{p(x) p(i)} q^{|x| \cdot i} x r_j({}_i r(y)) \\ 
	&+ \pi^{p(x) \cdot p(i) + p({}_i r(y))p(j)} q^{|x|\cdot i + |{}_i r(y)|\cdot j} r_j (x) {}_i r(y)
	\end{align*}
	and 
	\begin{align*}
	{}_i r \circ r_j (xy) = {}_i r(x) r_j(y) + \pi^{p(y) p(j)} q^{|y| \cdot j} & {}_i r(r_j(x)) y + \pi^{p(x) p(i)} q^{|x| \cdot i} x {}_i r(r_j(y)) \\ 
	&+ \pi^{p(y) \cdot p(j) + p(r_j(x))p(i)} q^{|y|\cdot j + |r_j (x)|\cdot i} r_j (x) {}_i r(y),
	\end{align*} 
	and since $ p(r_k(z)) = p(z) - p(k) $, the $\pi$ powers in the last term of each of the two expressions on the right are both equal to $ p(x) p(i) + p(y) p(j) - p(i) p(j) $; similarly $ | r_k (z) | = |z| - k $ so the $ q $ powers are both $ |x|\cdot i + |y| \cdot j - i \cdot j $, and so the two expressions agree by application of the inductive hypothesis.
\end{proof}

The following proposition from \cite{CHW13} is a key ingredient in the construction of the quasi-$K$-matrix:
\begin{prop}[Prop~2.2.2 of \cite{CHW13}]
	\label{prop:EF222}
	For $x\in \ffpr$ and $i\in I$, we have in $\UU'$
	\vspace{.1in}
	\begin{enumerate}
		\item[(a)] $\displaystyle
		x^+F_i-\pi_i^{p(x)}F_ix^+=\frac{r_i(x)^+\wtd{J}_i\wtd{K}_i-\wtd{K}_{-i}\,\,
			\pi_i^{p(x)-p(i)}\, _i r(x)^+}{\pi_iq_i-q_i^{-1}},
		$
		\vspace{.1in}
		\item[(b)] $\displaystyle
		E_ix^--\pi_i^{p(x)}x^-E_i=\frac{\wtd{J}_i\wtd{K}_{i}\,\,
			_i r(x)^- -\pi_i^{p(x)-p(i)} r_i(x)^-\wtd{K}_{-i}}{\pi_iq_i-q_i^{-1}}.
		$
	\end{enumerate}
\end{prop}

\subsection{\bf Bar-involution and Quasi-$\mathcal{R}$-matrix for $\UU$}
\label{subsec:bar+rmatrix}

There exists a unique $\Q$-algebra involution $\ov{\phantom{r}}$ (the bar-involution) on $ \Q(q)^\pi $ satisfying $\ov{q}= \pi q^{-1}$ and $\ov{\pi}= \pi$. For a bar-consistent super Cartan datum, 
\begin{align}\label{eq:barinvqi}
	\ov{q_i}=\pi_i
	q_i^{-1}.
\end{align}
Furthermore, there exists a bar-involution $\ov{\phantom{r}}:\ffpr\rightarrow\ffpr$
such that $\ov{\theta_i} =\theta_i$ for all $i\in I$ and $\ov{fx}
=\ov{f}\ov{x}$ for $f\in\Q(q)^\pi$ and $x\in \ffpr$. This extends to a unique homomorphism of
$\Q$-algebras $\ov{\white{x}}:\UU\rightarrow\UU$ such that
\[\ov{E_i}=E_i,\quad \ov{F_i}=F_i,\quad \ov{J_{\mu}}=J_{\mu},\quad
\ov{K_\mu}=J_{\mu}K_{-\mu},
\]
and $\ov{fx}=\ov{f}\ov{x}$ for all $f\in\Q(q)^\pi$ and $x\in\UU$. 

We remark here that our conventions for the comultiplication here are the same as in \cite{CHW13}:
\begin{align}
\Delta(E_i) = E_i\otimes 1 + \wtd{J}_i\wtd{K}_i\otimes E_i\qquad
&\Delta(F_i) = F_i\otimes \wtd{K}_{-i} + 1\otimes F_i\quad (\text{for } i\in
I),
\\
\Delta(K_\mu) =K_\mu\otimes K_\mu, \qquad
&\Delta(J_{\mu}) =J_{\mu}\otimes J_{\mu}\quad (\text{for }\mu\in Y).
\end{align}

Let $\widehat{\UU \otimes \UU}$ be the completion of the $\Q(q)^\pi$-modules $\UU \otimes \UU$ with respect to the descending sequence of subspaces 
\[
\UU^+ \UU^0 \big(\sum_{\hgt(\mu) \geq N}\UU_{\mu}^- \big) \otimes \UU  + \UU \otimes \UU^- \UU^0 \big(\sum_{\hgt(\mu) \geq N}\UU_{\mu}^+ \big) ,   \text{ for }N \ge 1, \mu \in \Z[\I].
\]
We have the obvious embedding of $\UU \otimes \UU$ into $\widehat{\UU \otimes \UU}$. 
%We let $\widehat{\UU}^+$ be the closure of $\UU^+$ in $\widehat{\UU}$, and so $\widehat{\UU}^+ \subseteq \widehat{\UU}$. 
By continuity the $\Q(q)^\pi$-algebra structure on $\UU \otimes \UU$ extends to a $\Q(q)^\pi$-algebra structure on $ \widehat{\UU \otimes \UU}$ cf. \cite[\S3.1]{CHW13}). Let $\ov{\phantom{x}}:\UU\otimes \UU\rightarrow \UU\otimes \UU$ be the $\Q$-algebra homomorphism given by $\ov{\phantom{x}}\otimes\ov{\phantom{x}}$. This extends to a $\Q$-algebra homomorphism on the completion. Let $\ov{\Delta}:\UU\rightarrow \UU\otimes\UU$ be the  $\Q(q)^\pi$-algebra homomorphism given by $\ov{\Delta}(x)=\ov{\Delta(\ov{x})}$. In \cite[\S~3.1]{CHW13}, the {\em quasi-$\mathcal R$-matrix} $\Theta$ for $\UU$ that intertwines $\Delta$ and $\ov{\Delta}$ is defined: For $\nu=\sum_i\nu_i i\in \N[I]$, write $\nu =\sum_{a=1}^{\hgt{\nu}} i_{a}$ for $i_a\in I$. Then, set $	e(\nu) =\sum_{a<b} p(i_a)p(i_b) \in \Z. $

\begin{prop}
	\label{prop:rmatrix}
		There is a unique family of elements $\Theta_\nu\in
		\UU^-_\nu\otimes\UU^+_\nu$ (with $\nu\in\N[I]$) such that
		\begin{enumerate}
		\item[(a)]
		$\Theta_0=1\otimes 1$ and $\Theta=\sum_\nu \Theta_\nu\in
		\widehat{\UU \otimes \UU}$ satisfies in $\widehat{\UU \otimes \UU}$ the identity
		$\Delta(u)\Theta=\Theta\ov{\Delta}(u)$ for all $u\in \UU$.
		
		\item[(b)]
		Let $B$ be a $\Q(q)^\pi$-basis of $\ff$ such that $B_\nu=B\cap
		\ff_\nu$ is a basis of $\ff_\nu$ for any $\nu$. Let $\{b^*| b\in
			B_\nu\}$ be the basis of $\ff_\nu$ dual to $B_\nu$ under the bilinear form$(\cdot,\cdot)$. Then,
		\[
		\Theta_\nu=(-1)^{\hgt(\nu)} \pi^{e(\nu)} \pi_\nu q_{\nu}\sum_{b\in B_\nu}
		b^-\otimes b^{*+}\in \UU^-_\nu\otimes\UU^+_\nu.
		\]
	\end{enumerate}
\end{prop}

We will use $\Theta$ in the construction of the quasi-$\mathcal R$-matrix for $\UUi$ in \S\ref{subsec:Thetai}.

\subsection{$\A$-form and modified form of $\UU$}
\label{subsec:UAmodified}
For $i\in I$, let $\theta^{(m)}_{i}$ denote the divided power $ \theta^m_{i}/[m]_i^!$ for  $m \ge 0$. Let $_\A\ff$ be the $\A^\pi$-subalgebra of $\ff$ generated by all divided powers $\theta^{(m)}_{i}$ for $m \ge 0$ and $i \in \I$. Under the identification of $ \ff $ with $ \UU^- $ sending $ \theta_i \mapsto F_i$, $\UU_{-\mu}^-$ can be identified with the image of $\ff_\mu$. Similarly, we can identify $\ff \cong \UU^+$ via $\theta_{i}$ with $E_{i}$. We let $_\A\UU^-$ (respectively, $_\A\UU^+$) denote the image of $_\A\ff$ under this isomorphism, which is generated by all divided powers $F^{(m)}_{i}$ (respectively, $E^{(m)}_{i}$). 

Recall from \cite[Definition 4.2]{CFLW} that the modified quantum covering group $\dot{\UU}$ is a non-unital $\mathbb{Q}(q)^{\pi}$-algebra generated by the symbols $1_{\lambda}$ (idempotents), $ E_i1_{\lambda}$ and $F_i1_{\lambda}$, for $\lambda\in X$
and $i \in I$ and with relations:
\begin{eqnarray*}
	&1_{\lambda} 1_{\lambda'} =\delta_{\lambda, \lambda'} 1_{\lambda},
	\vspace{6pt}\label{eq:modified idemp rel}\\
	&   (E_i1_{ \lambda })  1_{\lambda'} =  \delta_{\lambda, \lambda'} E_i1_{ \lambda}, \quad
	1_{\lambda'}   (E_i1_{\lambda}) = \delta_{\lambda', \lambda+ i'} E_i1_{\lambda},
	\vspace{6pt}\label{eq:modified E rel}\\
	& (F_i1_{ \lambda})  1_{\lambda'}=  \delta_{\lambda, \lambda'} F_i1_{ \lambda}, \quad
	1_{\lambda'}  (F_i1_{\lambda})=  \delta_{\lambda', \lambda- i'} F_i1_{\lambda},
	\vspace{6pt}\label{eq:modified F rel}\\
	& (E_iF_j-\pi^{p(i)p(j)}F_jE_i)1_{\lambda}=\delta_{ij}\left[{\langle i, \lambda\rangle}\right]_{v_i,\pi_i}1_\lambda,
	\vspace{6pt}
	\label{eq:modified comm rel}\\
	& \sum_{n+n'=1-\ang{i,j'}}(-1)^{n'}\pi_i^{n'p(j)+\binom{n'}{2}} E_i^{(n)} E_jE_i^{(n')} 1_{\lambda}=0
	\;\; (i\neq j),  
	\vspace{6pt}
	\label{eq:modified E Serre}\\
	&\sum_{n+n'=1-\ang{i,j'}}(-1)^{n'}\pi_i^{n'p(j)+\binom{n'}{2}} F_i^{(n)}F_jF_i^{(n')} 1_{\lambda}=0
	\;\;  (i\neq j),  
	\label{eq:modified F Serre}
\end{eqnarray*}
where $i,j \in I$, $\la, \la'\in X$, and we use the notation $xy1_\lambda=(x1_{\lambda+|y|})(y1_\lambda)$
for $x,y\in \UU$. The modified quantum covering group $\dot{\UU}$ admits an ${\A}^{\pi}$-form, ${}_{{\A}}\dot{\UU}$ and so we can define ${}_{R}\dot{\UU} = R^{\pi} \otimes_{{\A}^{\pi}} {}_{{\A}}\dot{\UU}$.

\cite[Lemma~3.5]{Cl14} goes here/after the following section.

\subsection{Canonical basis and based $\UU$-modules}

\label{subsec:CBforQCG}

Here we recount some terminology and background on canonical basis and based $ \UU $-modules.

Let $M(\lambda)$ be the Verma module of $\UU$ with highest weight $\lambda\in X$ and with a highest weight vector denoted by $\eta_{\lambda}$.  
Define a lowest weight $\UU$-module $^\omega M(\lambda)$ with the same underlying vector space as $M(\lambda)$ but with the action twisted by the involution $\omega$ given in \cite[\S2,.2]{CHW13}. 
We will denote the lowest weight vector $\eta_{\lambda}$ in $^\omega M(\lambda)$ by $\xi_{-\lambda}$.
Let $X^+ = \{\lambda \in X \mid \langle i, \lambda \rangle \in {\N}, \forall i \in \I  \} $ be the set of dominant integral weights. 
By $\la \gg 0$ we shall mean that the integers $\langle i, \lambda \rangle$ for all $i$ are sufficiently large. The Verma module $M(\lambda)$ associated to $\la \in X$ has a unique simple quotient $\UU$-module, denoted by $L(\lambda)$. We shall abuse the notation and denote by $\eta_\lambda \in L(\lambda)$ the image of the highest weight vector $\eta_\lambda \in M(\lambda)$. Similarly we define the $\UU$-module $^\omega L(\lambda)$ of lowest weight $-\la$ with lowest weight vector $\xi_{-\la}$.  
For $\la \in X^+$, we let $_\A L(\la) ={_\A\UU^-}  \eta_\la$ and $^\omega _\A L(\la) ={_\A\UU^+ } \xi_{-\la}$ be the $\A$-submodules of $L(\la)$ and $^\omega L(\la)$, respectively. 

We recall now the canonical basis for the half-quantum group developed in \cite{CHW14}: Let $R$ be a ring. A {\em $\pi$-basis} for a free $R^\pi$-module $M$ is a set $S\subset M$ such that there exists an $R^\pi$-basis $B$ for $M$ with $S= B\cup \pi B$.
Note that in \cite{CHW14}, this is called a {\em maximal $\pi$-basis}.
We note that a $\pi$-basis of an $R^\pi$-module $M$ is an $R$-basis of $M$.
The fundamental result on $\pi$-bases in {\em loc. cit.} is the following.

\begin{prop}[\cite{CHW14}]
	\label{prop:qcgcb}
	There is a $\pi$-basis $\cB$ of $\ff$ with the following properties:
	\begin{enumerate}
		\item $\cB$ is a $\pi$-basis of $\ff$ over $\A$.
		\item Each $b\in \cB$ is homogeneous.
		\item $\bar{b}=b$ for all $b\in \cB$.
		\item For $\lambda\in X^+$, there is a subset $\cB(\lambda)$  such that
		$\cB(L(\lambda))=\{b\eta_\lambda:b\in \cB(\lambda)\}$ is a $\pi$-basis of $L(\lambda)$, and if $b\in \cB\setminus \cB(\lambda)$, 
		$b^- \eta_\lambda=0$.
		%\item $\cB$ spans a crystal lattice $\cL$ 
		%of $\UU^-$
		%with crystal basis $\cB+v\cL$ .
		%\item $\cB(L(\lambda))$  spans a crystal lattice $\cL(\lambda)$ of $L(\lambda)$
		%with crystal basis $\cB(L(\lambda))+v\cL(\lambda)$.
	\end{enumerate}
\end{prop}

We note that $\cB|_{\pi=1}\subset\ff|_{\pi=1}$ is precisely the Lusztig-Kashiwara canonical basis. 
%see also: \S14.2 of \cite{Lu} for Lusztig's signed basis for $\ff$

Thus, there is a canonical basis $\{b^+|b\in \B\}$ on $\UU^+$, and a canonical basis $\{b^-|b\in \B\}$ on $\UU^-$. For each $\la\in X^+$, there is a subset $\B(\la)$ of $\B$ so that $\{b^- \eta_\la |b\in \B(\la)\}$ (respectively, $\{b^+\xi_{-\la} |b\in \B(\la)\}$)  forms a canonical basis of $L(\la)$ (respectively, $^\omega L(\la)$).  For any Weyl group element $w\in W$, let $\eta_{w\la}$ denote the unique canonical basis element of weight $w\la$. 

Let $\UUdot $ be the idempotented modified quantum group and ${}_\A \UUdot$ its $\A$-form. Then the sets $\{b^+1_\lambda b'^-: (b,b')\in\cB\times_\pi \cB\}$ and
$\{b^-1_\lambda b'^+: (b,b')\in\cB\times_\pi \cB\}$ both form a $\pi$-basis
of $_\A \UUdot$ (cf. \cite[Lemma~3.5]{Cl14}) and $\UUdot$ admits a canonical basis $\dot{\B}  = \{ b \diamondsuit_{\zeta} b'  \vert (b, b') \in \B \times \B, \zeta \in X \}$ (cf. \cite[Corollary~4.15]{Cl14}).

%%
%\subsection{Based $\UU$-modules}

Recall the notion of based modules for finite type quantum groups \cite[chapter~27]{Lu94}, and generalized to quantum groups of Kac-Moody type in \cite{BW16}, which is a module with a distinguished basis and compatible bar-involution. Like many results for quantum groups, these generalize to the quantum covering setting, see \cite[\S4]{Cl14} and also section~\ref{sec:iCBmodules} where we define based modules for $\UUi$. Examples of based $\UU$-modules include $L(\la)$ and $ ^\omega L(\la) $ with their $\pi$-basis $\B(\la) $. We go through a few relevant results here:

\begin{prop}
	\label{prop:CBtensor}
	%Assume that the root system is $Y$-regular. 
	Let $(M,B), (M',B')$ be based modules, with either $M = ^\omega L(\la)$ or $M'=L(\la)$ for $\la \in X^+$. 
	Let $\mc L$ be the $\Z^\pi[q^{-1}]$-submodule of $M\otimes M'$ generated by $B\otimes B'$.
	
	\begin{enumerate}
		\item
		For any $(b, b')\in B\times B'$, there is a unique element $b\diamondsuit b'\in \mc L$ such that 
		$\Psi(b\diamondsuit b') =b\diamondsuit b'$ and $b\diamondsuit b' -b\otimes b' \in q^{-1} \mc L$.
		
		\item
		The element $b\diamondsuit b'$ is equal to $b\otimes b'$ plus a $q^{-1} \Z^\pi[q^{-1}]$-linear combination of
		elements $b_2\otimes b_2'$ with $(b_2,b_2') \in B\times B'$ with $(b_2,b_2') <(b,b')$. 
		
		\item
		The elements $b\diamondsuit b'$ with $(b,b') \in B\times B'$ form a $\Q(q)^\pi$-basis of $M\otimes M'$,
		an $\A^\pi$-basis of $\A^\pi \otimes_{\Z^\pi[q^{-1}]} \mc L$, and a $\Z^\pi[q^{-1}]$-basis of $\mc L$. 
	\end{enumerate}
\end{prop}
 
\begin{proof}
	The argument here is a direct generalization of \cite[Theorem~2.7]{BW16},  using the quasi-$\mc R$-matrix $ \Theta $ from \S\ref{subsec:bar+rmatrix} above and a similar construction to \cite[Corollary~4.2]{Cl14}.
\end{proof}

By applying this iteratively we have generalization of \cite[Proposition~2.9]{BW16} and direct generalizations of constructions in \cite[\S27]{Lu94} to the quantum covering setting leads to the quantum covering analogue of \cite[Prop~2.11]{BW16}:
%This follows from Remark 2.11 of loc. cit., and direct generalizations of versions of Lu 27.1.7-8 and 27.2.1-2 to the quantum covering setting (omit these details, nothing particularly novel is introduce here). 

\begin{prop}
	Let $\la_1, \ldots, \la_\ell \in X^+$. Let $\eta_i$ denote the highest weight vector of $L(\la_i)$ for each $i$ and let $\eta$ denote the highest weight vector of $L(\sum_{i=1}^\ell \la_i)$.
	Then the (unique) homomorphism of $\UU$-modules 
	$$
	\chi: L\Big(\sum_{i=1}^\ell \la_i \Big) \longrightarrow L(\la_1) \otimes \ldots \otimes L(\la_\ell),
	\qquad \chi(\eta) =\eta_1\otimes \ldots \otimes \eta_\ell
	$$ 
	sends each canonical basis element to a canonical basis element.  
\end{prop}

For $ \lambda, \mu \in X^+ $  we define the $ \UU $-submodule $ L(\lambda, \mu) := \UU(\eta_{\lambda} \otimes \eta_\mu) \subset L(\lambda) \otimes L(\mu) $.

\begin{prop}
	\label{thm:basedsub}
	Let $ \lambda, \mu \in X^+ $ and $ w \in W $. Then, the $ \UU $-submodule $ L(\lambda,\mu) $ is a based $\UU$-submodule of $ L(\lambda) \otimes L(\mu) $.
\end{prop}

\begin{proof}
	Write $ \lambda = \lambda_1 - \nu $. From the results above, $ L(\lambda_1) \otimes L(\mu) $ is a based $ \UU $-module, and the map  $ \chi : L(\lambda_1 + \mu) \to L(\lambda_1) \otimes L(\mu) $ is a based $ \UU $-module homomorphism, and so $ \chi' := id_{{}^\omega L(\lambda)} \otimes \chi $ is a based module homomorphism. Similarly, the map $ \phi : {}^\omega L(\nu) \otimes L(\lambda_1) \to L(\lambda) $ is a based module homomorphism, and hence so is $ \phi' := \phi \otimes \id_{L(\mu)} $. Thus, the composition homomorphism $ \phi' \chi': {}^\omega L(\nu) \otimes L(\lambda_1 + \mu) \to L(\lambda) \otimes L(\mu) $ sending $ \xi_\nu \otimes \eta_{\lambda_1 + \mu} \mapsto \eta_{\lambda} \otimes \eta_\mu $ is a based $ \UU $-module homomorphism. Since $ {}^\omega L(\nu) \otimes L(\lambda_1 + \mu) $ is cyclically generated by $\eta_{\lambda} \otimes \eta_\mu$, the $ \UU $-module $ L(\lambda, \mu) $ is the image of the based module homomorphism $ \phi' \chi' $, and so $ L(\lambda,\mu) $ is a based $ \UU $-submodule of $ L(\lambda) \otimes L(\mu) $.
\end{proof}

\section{The $\imath$quantum covering groups $\UUi$}
 \label{sec:Ui}
  
  %todo{review notation}
  
  We begin with a definition (cf. \cite[Definition~2.2]{C19}):
  
  \begin{definition}
  	\label{def:UUi}
  	The {\em quasi-split $\imath${}quantum covering group}, denoted by $\UUi_{\vs}$ or  just $\UUi$, is the $\K(q)^\pi$-subalgebra of $\UU$ generated by
  	\begin{align}
  	B_i :=F_{i}  &+ \vs_i E_{\tau i} \tK^{-1}_i, \qquad \tJ_i \, \,(i \in I),
  	\qquad K_{\mu}\, \,(\mu \in Y^{\imath}).
  	\label{eq:def:ff}
  	\end{align}
  	Here the parameters
  	\begin{equation}
  	\label{parameters}
  	\vs=(\vs_i)_{i\in I}\in ((\K(q)^\pi)^\times)^I,\qquad
  	% 	\kappa=(\kappa_i)_{i\in I}\in\K(q)^I
  	\end{equation}
  	are assumed to satisfy Conditions \eqref{bar1}--\eqref{bar3} below:
  	\begin{align}
  	\ov{\vs_iq_i} &= \vs_iq_i \text{ if } \tau i=i \text{ and } a_{ij}\neq 0 \text{ for some } j\in I\setminus\{i\}; \label{bar1} \\
  	\ov{\vs_i} &= \vs_i = \vs_{\tau i}, \text{ if } \tau i\neq i \text{ and } a_{i,\tau i}=0.	
  	\label{bar2} \\
  	\vs_{\tau i} &= \pi_i q_i^{-a_{i,\tau i}} \overline{\vs_i} \text{ if } \tau i\neq i \text{ and } a_{i,\tau i} \neq 0.   
  	\label{bar3}
  	\end{align}
  \end{definition}
  
  The $\imath${}quantum covering group is a (right) coideal subalgebra of $\UU$, since under the comultiplication $ \Delta : \UUi \to \UUi \otimes \UU $. We will occasionally denote the embedding by  $\imath : \UUi \hookrightarrow \UU$; the $\imath$ in the name  and superscript originates from this convention. The conditions on the parameters ensure that $\UUi$ admits a suitable bar-involution (see \S\ref{subsec:bar}).
  
  \subsection{The $\imath^\pi$ divided powers}
  \label{subsec:piDP}
 
 Let $ \UUi = \UUi_\vs $ be an $\imath{}$quantum group with parameter $ \vs $, for a given root datum $ (Y,X,\ang{\cdot,\cdot},\ldots )$.
 
 \begin{definition}
 	For  $i\in I$ with $\tau i\neq i$, imitating Lusztig's divided powers,  we define the {\em divided power} of $B_i$ to be
 	\begin{align}
 		\label{eq:iDP1}
 		B_i^{(m)}:=B_i^{m}/[m]_{i}^!, \quad \forall m\ge 0, \qquad \text{when } i \neq \tau i.
 	\end{align}
 	% Observe that $B_i^{(m)}$ lies in ${}_\A \U$ thanks to the $q,\pi$-binomial formula. \red{Here $_\A \U$ is the $\A$-subalgebra of $\U$ generated by the elements $E_i^{(m)}$, $F_i^{(m)}$ for various $i\in I$ and $m\in\Z$ and by the elements $K_\mu$ for $\mu\in Y$.}
 	
 	For $ i \in I $ with $ \tau i = i $, the {\em $\imath^\pi$-divided powers} are defined to be
 	\begin{eqnarray}
 		&& B_{i,\odd}^{(m)}=\frac{1}{[m]_{i}^!}\left\{ \begin{array}{ccccc} B_i\prod_{j=1}^k (B_i^2-\vs_i \pi_i q_i [2j-1]_{i}^2 \tJ_i ) & \text{if }m=2k+1,\\
 			\prod_{j=1}^k (B_i^2- \vs_i \pi_i q_i [2j-1]_{i}^2 \tJ_i ) &\text{if }m=2k; \end{array}\right.
 		\label{eq:piDPodd}\\
 		&& B_{i,\even}^{(m)}= \frac{1}{[m]_{i}^!}\left\{ \begin{array}{ccccc} B_i\prod_{j=1}^k (B_i^2- \vs_i q_i  [2j]_{i}^2 \tJ_i ) & \text{if }m=2k+1,\\
 			\prod_{j=1}^{k} (B_i^2- \vs_i q_i  [2j-2]_{i}^2 \tJ_i ) &\text{if }m=2k. \end{array}\right.
 		\label{eq:piDPev}
 	\end{eqnarray}
 \end{definition}
 
 When we specialize $ \pi_i = 1 $ and $ \tJ_i = 1 $, we obtain the $\imath$-divided powers in \cite{CLW18} from the formulas above.. These $ \imath^\pi $-divided powers satisfy closed form expansion formulas when written in terms of the PBW basis for $ \UU $ (see \cite[\S3.3--3.7]{C19}), which enables the formulation of a Serre presentation for $ \UUi $ in \cite[Theorem~4.2]{C19}, generalizing \cite[Theorem~3.1]{CLW18}.

  \subsection{Bar involution on $\UUi$}
  \label{subsec:bar} 
  One application of the Serre presentation for $ \UUi $ is that it enables us to establish the existence of the bar involution for the quasi-split $\imath${}quantum group $\UUi$ in \cite[Prop~4.10]{C19}:
  
  \begin{prop}
  	\label{prop:bar}
  	Assume the parameters $\vs_i$, for $i\in I$, satisfy the conditions \eqref{bar1}--\eqref{bar3} above. 
  		Then there exists a $\K$-algebra automorphism $ ^{\ov{\,\,\,\,\,}\imath}: \UUi\rightarrow \UUi$ (called a bar involution on $\UUi$) such that
  	\[
  	\ov{q}^\imath=\pi q^{-1}, \quad
  	% i.e. \ov{q}^\imath=\ov{q}
  	\ov{B_i}^\imath=B_i,  \quad
  	\ov{\tJ_i}^\imath=\tJ_i,  \quad
  	\ov{K_\mu}^\imath=J_\mu K_{\mu}^{-1}, \quad
  	\forall \mu\in Y^\imath, i\in I.
  	\]
  \end{prop}
  
  Note that bar-involution for $\UUi$ (which we will henceforth denote with $\ipsi$) differs from the bar-involution for $\UU$ (which we will now call $\psi$) defined in \S~\ref{subsec:bar+rmatrix} previously when restricted to $\UUi$: $ \ipsi $ fixes $ B_i $ but $ \psi(F_i + \varsigma_i E_{\tau i} \tK_i^{-1}) = F_i + \ov{\varsigma_i} E_{\tau i} \tJ_i \tK_i $
  In the next section, we will construct a quasi-$K$-matrix $ \Upsilon $ intertwining the two involutions, which will lead to a theory of canonical bases in the following sections.

%%%%%%%%
%%%%%%%%
\section{Quasi-$K$-matrix}
\label{sec:quasiKmatrix}
% Remark on convention: Even though BW uses leading term E (Kashiwara's notation is more compatible with K-matrices etc), using leading term F is more useful for e.g. Lusztig's braid group operator (QSG V).

The goal of this section will be the development of a quasi-$K$-matrix for $\UUi$. Let $\hat{\UU}$ be the completion of $ \UU $ with respect to the descending sequence of $\Q(q)^\pi$-submodules $\UU^-\UU^0\left(\sum_{\hgt(\mu)\geq N}\UU^+_\mu\right)$. We have an embedding of $ \UU $ into $ \hat{\UU} $,  and by continuity the $ \Q(q)^\pi $-algebra structure on $ \UU $ extends to $ \hat{\UU} $, and the bar-involution $\psi$ on $ \UU $ extends to an involution on $ \hat{\UU} $, which we also denote $ \psi $. Let $ \hat{\UU}^+ $ denote the closure of $ \UU^+ $ in $ \hat{\UU} $. 

We will show that there exists a unique family of elements $ \Upsilon_\mu \in \UU_\mu^+ $ such that $ \Upsilon_0 = 1 $ and $ \Upsilon = \sum_\mu \Upsilon_\mu $ satisfies the following identity in $ \hat{\UU} $:
	\begin{equation}
		\label{eq:intertwiner}
		\psi_\imath(u) \Upsilon = \Upsilon \psi(u), \quad \text{for all } u \in \UUi.
	\end{equation}
$\Upsilon$ is called the {\em quasi-$K$-matrix} cf. \cite{BK15}; the terminology {\em intertwiner} also appears in the literature e.g. \cite[Chapter~2]{BW18a}, since $ \Upsilon $ "intertwines" the bar-involutions $ \psi^\imath $ for $ \UUi $ and $ \psi $ for $ \UU $, which are not compatible under the embedding $\imath$. 

\subsection{A parity operator}
A crucial ingredient of the quasi-$K$-matrix construction in \cite{BW18a} is \cite[Prop~3.1.6]{Lu94}; its quantum covering analogue is Proposition~\ref{prop:EF222} above. However, when attempting a similar computation in the quantum covering case, we run into the following issue: since $ \overline{B_i} = F_i + \overline{c_i} E_i J_i K_i $ in $ \UU$, we would like to have $ \Upsilon = \sum_{\mu} \Upsilon_{\mu} \in \hat{\UU}^{+} $ satisfying
$$
(F_i + c_i E_i K_i^{-1}) \Upsilon = \Upsilon (F_i + \overline{c_i} E_i J_i K_i)
$$
and so we have equivalently that
$$
F_i \Upsilon_{\mu} - \Upsilon_{\mu} F_i = \Upsilon_{\mu-2i} \overline{c_i} E_i J_i K_i - c_i E_i K_i^{-1} \Upsilon_{\mu-2i}
$$
Unfortunately here we cannot apply Prop~\ref{prop:EF222} when $ p(\mu) = \bar{1} $ due to an extraneous factor of $ \pi_i $. 

Borrowing inspiration from \cite{BKK}, we can get around this issue by enlarging our algebra slightly by introducing a parity operator $\sigma$ such that
$$
\sigma E_i = \pi^{p(i)} E_i \sigma, \quad \sigma F_i = \pi^{p(i)} F_i \sigma, \quad \sigma K_\mu = K_\mu \sigma \text{ and } \sigma J_\mu = J_\mu \sigma
$$
and separating odd and even parts $\Upsilon = \Upsilon_{\bar{0}} + \sigma \Upsilon_{\bar{1}} $.

\subsection{Quasi-$K$-matrix for $\osp(1|2n)$}
\label{sec:osp12n}
In finite type rank $n$, we want to define $ \Upsilon = \sum_{\mu} \sigma^{p(\mu)} \Upsilon_{\mu} \in \hat{\UU}^{+} $ satisfying
$$
(F_i + c_i E_i K_i^{-1}) \Upsilon = \Upsilon (F_i + \overline{c_i} E_i J_i K_i)
$$
which together with Proposition~\ref{prop:EF222} yields equivalent conditions (which are the same for $ p(\mu) $ even or odd) in terms of the twisted derivations $r_i $ and $ {}_i r $ defined as in \S~\ref{subsec:rtwisted}:
\begin{align}
\label{riup}
r_i (\Upsilon_\mu) &= -(\pi_i q_i - q_i^{-1})(c_i \pi_i q_i^2 ) \Upsilon_{\mu - 2 i} E_i  \\
\label{irup}
{}_i r(\Upsilon_\mu) &= -(\pi_i q_i - q_i^{-1})(c_i \pi_i q_i^2 ) E_i \Upsilon_{\mu - 2 i}
\end{align}
where we have used the fact that $ \pi_i^{p(i)} = \pi^{p(i)^2} = \pi_i $ since by the bar-consistency condition $p(i) \equiv d_i \pmod 2$.

With this, we can use the methods in \cite[Section~2.4]{BW18a} (cf. also \cite[Section~6.2]{BK18}) to construct $ \Upsilon $. Recall the non-degenerate symmetric bilinear form $ (\cdot,\cdot) $ on $ '\bf{f} $ defined preceding \S~\ref{subsec:rtwisted} above. We have
\begin{align*}
\eqref{riup} & \iff	(r_i (\Upsilon_\mu), z) = -(\pi_i q_i - q_i^{-1})(c_i \pi_i q_i^2 ) (\Upsilon_{\mu - 2 i} E_i, z) \\
& \overset{\eqref{eq:derivadjunct}}{\iff}  (\Upsilon_\mu, z E_i) = -(\pi_i q_i - q_i^{-1}) (c_i \pi_i q_i^2 ) (E_i,E_i)^2 (\Upsilon_{\mu - 2 i}, r_i(z)) \\
& \iff (\Upsilon_\mu, E_i z) = - c_i q_i^3 (1 - \pi_i q_i^{-2})^{-1} (\Upsilon_{\mu - 2 i}, {}_i r(z)) .
\end{align*}
Applying a similar argument to $ \eqref{irup}$, we have:
\begin{lem}
The conditions \eqref{riup}--\eqref{irup} yield the equivalent conditions
\begin{align}
	\label{pupl}
	(\Upsilon_\mu, E_i z) &= - c_i q_i^3 (1 - \pi_i q_i^{-2})^{-1} (\Upsilon_{\mu - 2 i}, {}_i r(z)) \\
	\label{pupr}
	(\Upsilon_\mu, z E_i) &= - c_i q_i^3 (1 - \pi_i q_i^{-2})^{-1} (\Upsilon_{\mu - 2 i}, r_i(z))
\end{align}
\end{lem}

Thus we may inductively define $ \Upsilon^*_L $ and $ \Upsilon^*_R $ in $ '\bf{f}^* $ the non-restricted dual of $ '\bf{f} $ such that $ \Upsilon^*_L(1) = \Upsilon^*_R(1) = 1 $ and
\begin{align}
\label{upl}
\Upsilon^*_L(E_i z) &= - c_i q_i^3 (1 - \pi_i q_i^{-2})^{-1} \Upsilon^*_L ({}_i r(z)) \\
\label{upr}
\Upsilon^*_R(z E_i) &= - c_i q_i^3 (1 - \pi_i q_i^{-2})^{-1} \Upsilon^*_R(r_i(z))
\end{align}
\\
Note that for all $ i,j \in I $, we have from $ {}_i r(1) = 0 $ and $ {}_i r(E_j) = \delta_{ij} $ that
$$
\Upsilon^*_L(E_i) = 0\quad\text{and}\quad \Upsilon^*_L(E_iE_j) = - c_i q_i^3 (1 - \pi q_i^{-2})^{-1}\delta_{ij},
$$
and similarly for $ \Upsilon^*_R $.

\begin{lem}
	\label{lem:evenups}
	For $ x \in \ffpr_\mu $, if either $ p(\mu) $ or $ \hgt(\mu) $ is odd, then $ \Upsilon^*_L(x) = \Upsilon^*_R(x) = 0 $ 
\end{lem}
\begin{proof}
	We show this for odd $ p(\mu) $ by induction on $\hgt(\mu)$ (the statement for odd $ \hgt(\mu) $ is similar). The base cases $ \hgt(\mu) = 1,3 $ are given above. For homogeneous such $ x \in \ffpr_\mu $, $ x = E_i z $ for some $ z \in \ffpr_\nu $ so $ {}_i r(z) \in \ffpr_{\nu - i} $ where $ p(\nu - i) $ is odd ($p(\nu)$ and $p(i)$ have opposite parity since $ p(\mu) = p(\nu) + p(i) $ is odd), and so by induction hypothesis, $\Upsilon^*_L ({}_i r(z)) = 0 $, and hence by \eqref{upl}, $\Upsilon^*_L(x) = 0 $ as well (similarly for $ \Upsilon^*_R $).
\end{proof}

Note that as a result, there will be no odd terms in $ \Upsilon $ i.e. for $ p(\mu) = \ov{1} $, $ \Upsilon_\mu = 0 $. 

%similarly to \cite[Lemma~2.16]{BW18a}:
\begin{lem}
	\label{upequal}
	We have $ \Upsilon^*_L = \Upsilon^*_R $.
\end{lem}
\begin{proof}
	We will show that $ \Upsilon^*_L(x) = \Upsilon^*_R(x) $ for all homogeneous $ x \in \ffpr_{\mu} $ by induction on $\hgt(\mu)$, using Lemma~\ref{lem:rivanishing} above.

	The base cases $ \hgt(|x|) = 0 $ or $1$ are trivial from the definition. Suppose that the identity holds for all homogeneous elements with height no greater than $k$ for $ k \geq 1 $, and let $ x = E_i x' E_j $ with $ \hgt(|x|) = k+1 \geq 2 $ for some $ i,j \in I $. Let $ \xi_k = - c_k q_k^3 (1 - \pi_k q_k^{-2})^{-1} $. Then,
	\begin{align*}
	\Upsilon^*_L (E_i x' E_j) &= \xi_i \Upsilon^*_L({}_i r(x' E_j)) \\
	&= \xi_i \left( \Upsilon^*_L({}_i r(x') E_j) + \pi^{p(x') p(i)} q^{|x'|\cdot i} \Upsilon^*_L(x' {}_i r(E_j))  \right) 
	\end{align*}
	and 
	\begin{align*}
	\Upsilon^*_R (E_i x' E_j) &= \xi_j \Upsilon^*_R(r_j( E_i x')) \\
	&= \xi_j \left( \Upsilon^*_R(E_i r_j(x')) + \pi^{p(x') p(j)} q^{|x'|\cdot j} \Upsilon^*_R(x' r_j(E_i))  \right).
	\end{align*}
	
	The second terms of both of the final expressions above vanish unless $ i = j $, in which case they are both equal (by application of the induction hypothesis to $x'$ of height $ k - 1 $), so it remains to show that
	$$
	\xi_i \Upsilon^*_L({}_i r(x') E_j) = \xi_j \Upsilon^*_R(E_i r_j(x'))
	$$
	
	This can be done by applying the induction hypothesis to $ {}_i r(x') E_j $ and $ E_i r_j(x') $ to obtain
	$$
	\xi_i \Upsilon^*_L({}_i r(x') E_j) = \xi_i \Upsilon^*_R({}_i r(x') E_j) \overset{\eqref{upr}}{=} \xi_i \xi_j \Upsilon^*_R(r_j \circ {}_i r(x'))
	$$
	and
	$$
	\xi_j \Upsilon^*_R(E_i r_j(x')) = \xi_j \Upsilon^*_L(E_i r_j(x')) \overset{\eqref{upl}}{=} \xi_i \xi_j \Upsilon^*_L({}_i r \circ r_j(x'))
	$$
	and from the fact that $ r_j \circ {}_i r = {}_i r \circ r_j $ by Lemma \ref{lem:rijr}, and the induction hypothesis, since $ r_j \circ {}_i r(x') = {}_i r \circ r_j(x') \in \ffpr_{|x'| - i - j} $, the desired result follows.
\end{proof}

Thus, we can denote $ \Upsilon^*_L = \Upsilon^*_R $ by $  \Upsilon^* $. For the Serre relators $ S_{ij} $ for $ \osp(1|2n) $, $ |S_{ij}| $ has height $3$ when $ (i,j) \neq (n,n-1) $, and $ p(S_{n,n-1}) $ is odd, so by \ref{lem:evenups}:
$$
\Upsilon^*(S_{ij}) = 0
$$
and by the same induction argument in \cite[Lemma~2.17]{BW18a},
$$
\Upsilon^*(I) = 0
$$
where $ I = \langle S_{ij} \rangle $ (cf. \cite[Lemma~2.17]{BW18a}) so $ \Upsilon^* $ is an element in $ (\UU^+)^* $ (the unrestricted dual of $ \UU^+ $).

Then, we can construct $ \Upsilon $ in the same way [following the proof of Theorem 2.10]:

Let $ B = \{ b \} $ be a basis of $ \UU^- $ such that $ B_\mu = B \cap \UU_{-\mu}^+ $ is a basis for $ \UU_{-\mu}^+ $, and let $ B^* = \{ b^* \} $ be the dual basis of $B$ with respect to $ (\cdot,\cdot) $ and let 
\begin{equation}
\Upsilon := \sum_{b \in B} \Upsilon^*(b^*) b = \sum_{\mu} \Upsilon_\mu \in \hat{\UU}^+ .
\end{equation}
%observe: no terms with $ p(\mu) = 1 $. 
As functions on $ \UU^+ $, we have $ (\Upsilon,\cdot) = \Upsilon^* $, and $ \Upsilon_0 = 1 $. Also $ \Upsilon $ satisfies the identities in \eqref{riup} and \eqref{irup} by construction, because $ \Upsilon^* $ satisfies the equivalent identities in \eqref{upl} and \eqref{upr}). 

From this we see that $ r_i(\Upsilon_\mu) $ is determined by $ \Upsilon_\nu $ with weight $ \nu \prec \mu $. Together with Lemma~\ref{lem:rivanishing},
%(cf. \cite[Lemma~1.2.15]{Lu94})
this implies the uniqueness of $ \Upsilon $.\qed

\begin{rem}
	For rank 2, we can generalize the above slightly by having $ \kappa_1 \neq 0 $ (i.e. $ B_1 = F_1 + c_1 E_1 K_1^{-1} + \kappa_1 K_1^{-1} $). In this scenario, for $ i = 1 $ and $ \pi_1 = \pi^{p(1)} = 1 $, we have the following replacements for \eqref{riup} and \eqref{irup} (using $ \alpha_1 $ for $i = 1 \in I $ to avoid confusion):
\begin{align}
\label{r1up}
r_1 (\Upsilon_\mu) &= -(q_1 - q_1^{-1}) \left( (c_1 q_1^2 ) \Upsilon_{\mu - 2 \alpha_1} E_1 + \kappa_1 \Upsilon_{ \mu - \alpha_1 } \right) \\
\label{1rup}
{}_1 r(\Upsilon_\mu) &= -(q_1 - q_1^{-1}) \left( (c_1 q_1^2 ) E_1 \Upsilon_{\mu - 2 \alpha_1} + \kappa_1 \Upsilon_{ \mu - \alpha_1 } \right) 
\end{align}
This leads to the following replacements for $i = 1$ in the inductive definition for $ \Upsilon_L^* $ and $ \Upsilon_R^* $:
\begin{align}
\label{uplk}
\Upsilon^*_L(E_1 z) &= - c_1 q_1^3 (1 - q_1^{-2})^{-1} \Upsilon^*_L ({}_1 r(z)) - \kappa_1 q_1 \Upsilon_L^*(z) \\
\label{uprk}
\Upsilon^*_R(z E_1) &= - c_1 q_1^3 (1 - q_1^{-2})^{-1} \Upsilon^*_R(r_1(z)) - \kappa_1 q_1 \Upsilon_R^*(z).
\end{align}
It can then be checked that $ \Upsilon_L^* = \Upsilon_R^* =: \Upsilon^* $ and $ \Upsilon^*(I) = 0 $ for $ I = \langle S_{12}, S_{21} \rangle $, and so the above construction for $ \Upsilon $ also holds. Note that $ \Upsilon $ is still even in this case since $ \kappa_1 $ is a coefficient for the long, even root.
\end{rem}

\subsection{Example: rank $1$ (single odd root)} Let
$$
\Upsilon = \sum_{k \geq 0} a_{2k} E^{(2k)} + a_{2k+1} \sigma E^{(2k+1)}
$$
\\
Then, Proposition~\ref{prop:EF222} in rank one gives
\begin{equation*}
	\label{eq:2221}
	E^{(N)} F - \pi^{N} F E^{(N)} = \pi^{N-1} \qbinom{ K; 1-N }{1} E^{(N-1)} = \pi E^{(N-1)} \frac{ (\pi q)^{1-N} JK - q^{N-1} K^{-1} }{\pi q - q^{-1}}
\end{equation*}
\\
We need to separate the computation for the condition $ B \Upsilon = \Upsilon \overline{B} $ when $N$ is even from when $N$ is odd. When $N=2k$ is even, we have 
$$ 
a_{2k} ( E^{(2k)} F - \pi^{N} F E^{(2k)} ) = a_{2k - 2} ( cq^2 K^{-1} E E^{(2k-2)} - \bar{c} E E^{(2k-2)} JK )
$$ 
and so using \eqref{eq:2221} and comparing coefficients of $ E^{(2k-1)} JK $ and $ E^{(2k-1)} K^{-1} $ respectively yield the (over-determined) system of solutions
$$
a_{2k} = - c \pi q^2 (\pi q - q^{-1}) q^{1-2k} [2k-1] a_{2k-2}
$$
and
$$
a_{2k} = - c \pi q^2 (\pi q - q^{-1}) q^{2k-1} q^{2(1-2k)} [2k-1] a_{2k-2}.
$$
Hence for $k$ even,
\begin{align*}
	a_{2k} &= (-c \pi q^2)^k ( \pi q - q^{-1})^k q^{-k^2} [2k-1]^{!!}
\end{align*}
where $ [2k-1]^{!!} = [2k-1]\cdot [2k-3] \cdot \ldots \cdot [1] $ (normalization: $ a_0 = 1 $).
\\
For $N$ odd, we also obtain an over-determined system of two solutions:
\begin{align*}
	a_{2k+1} &= (-c \pi q^2) (\pi q - q^{-1}) q^{-2k} [2k] a_{2k-1} \\
	&= (-c \pi q^2)^{k+1} ( \pi q - q^{-1})^{k+1} q^{-2\binom{k+1}{2}} [2k]^{!!} a_{-1}
\end{align*}
where $[2k]^{!!} = [2k]\cdot [2k-2] \cdot \ldots \cdot [2]  $. Since $ a_{-1} = 0 $, we see that $ \Upsilon $ has no odd part.
\\
\\
So we have
\begin{equation*}
	\Upsilon = \sum_{k \geq 0} (-c \pi q^2)^k ( \pi q - q^{-1})^k q^{-k^2} [2k-1]^{!!} E^{(2k)} %+ \sigma  (-c \pi q^2)^k ( \pi q - q^{-1})^k q^{-2\binom{k+1}{2}} [2k]^{!!} a_1 E^{(2k+1)} 
\end{equation*}
%todo: comment: $\sigma$ does not appear in $\Upsilon$ i.e. it is even}
\\
\\
Note that $\Upsilon$ is a solution to the system of equations 
\begin{equation}
	\label{1r}
	{}_1r(\Upsilon) = -c \pi q^2 (\pi q - q^{-1}) E \Upsilon,
\end{equation}
and
\begin{equation}
	\label{r1}
	r_1(\Upsilon) = -c \pi q^2 (\pi q - q^{-1}) \Upsilon E,
\end{equation}
 and indeed may be defined as the unique such solution (cf. \cite[Proposition~6.3]{BK18})
\\
\\
Existence: this can be verified for $ \Upsilon $ defined above using $ {}_1 r(E^{(2k)}) = q^{2k-1} E^{(2k-1)} $ for the first equation:
\begin{align*}
	{}_1 r (\Upsilon_{2k}) &= {}_1 r (a_{2k} E^{(2k)}) \\
	&= a_{2k} q^{2k-1} E^{(2k-1)} \\
	&= - c \pi q^2 (\pi q - q^{-1}) a_{2k-2} [2k-1] \frac{EE^{(2k-2)}}{[2k-1]} \\
	&= (- c \pi q^2) (\pi q - q^{-1}) E \Upsilon_{2k-2},
\end{align*}
and using $ r_1(E^{(2k)}) = q^{2k-1} E^{(2k-1)} ( = {}_1 r(E^{(2k)}) ) $ for the second.
\\
\\
Note that this definition implies no odd part for $ \Upsilon $, because
\begin{align*}
	{}_1 r (\Upsilon_{2k+1}) &= {}_1 r (a_{2k+1} \sigma E^{(2k+1)}) \\
	&= a_{2k+1} q^{2k} \sigma E^{(2k)} \\
	&= - c \pi q^2 (\pi q - q^{-1}) a_{2k-1} \pi [2k] E \sigma \frac{ E^{(2k-1)}}{[2k]}% note that there is a $\pi$ from commuting $\sigma$ past E}} \\
	&= \pi (- c \pi q^2) (\pi q - q^{-1}) E \Upsilon_{2k-1}
\end{align*}

\begin{rem}[rank $1$ nonstandard]
	When we repeat the above computations with an additional term $ s K^{-1} $ in $B$, we get the condition that 
	$$
	a_N \left( {}_1 r (E^{(N)}) \right) = a_N \left( r_1 (E^{(N)}) \right) = - ( \pi q - q^{-1} ) ( c \pi q^2 [N-1] a_{N-2} + s \pi a_{N-1} \sigma E^{(N-1)} ),
	$$
	and since $ {}_1 r (E^{(N)}) = r_1 (E^{(N)}) = q^{N-1} E^{(N-1)}  $, there are no terms with $ \sigma E^{(N-1)} $ on the left hand side, and no solutions for $ s \neq 0 $.
\end{rem}

\subsection{Quasi-$K$-matrix for quasi-split QSP of general super Kac-Moody type}
Now let $ \UU $ be a general quantum covering group of super Kac-Moody type as defined in \S\ref{sec:QCG}, and $ (\UU,\UUi)  $ a quasi-split quantum symmetric pair for $ \UU $, with bar-involutions $\psi$ on $\UU$ and $\ipsi$ on $\UUi$ respectively.
\begin{thm}
	\label{thm:Upsilon}
	There exists a unique family of elements $ \Upsilon_\mu \in \UU_\mu^+ $ such that $ \Upsilon_0 = 1 $ and $ \Upsilon = \sum_\mu \Upsilon_\mu $ satisfies the following identity in $ \hat{\UU} $:
	\begin{equation}
%		\label{eq:intertwiner}
		\psi_\imath(u) \Upsilon = \Upsilon \psi(u), \quad \text{for all } u \in \UUi.
	\end{equation}
Morover, $ \Upsilon_\mu = 0 $ for all $ p(\mu) = \ov{1} $ .
\end{thm}
\begin{proof}
The constructions in \ref{sec:osp12n} are not particular to $ \osp(1|2n) $ and so hold for quasi-split $\UUi$ of general super Kac-Moody type with $ E_{\tau i} $ replacing $ E_i $, with the exception of checking that $ \Upsilon^*(S_{ij}) = 0 $ for general Serre relators. Using Remark~\ref{lem:evenups} we have shown that this is the case for ht$ (S_{ij}) $ odd, and so it remains to show this for ht$ (S_{ij}) $ even. This can be done term-wise i.e. by showing that terms of the form
\begin{equation}
\label{sijeven}
\Upsilon^*(E_i^a E_j E_i^b) \text{ for $ j \neq i $ and $ a+b+1 $ even}
\end{equation}
vanish. This can done by induction using \eqref{upl} or \eqref{upr}. For instance if $ a > 1 $, we may use \eqref{upl} to show that (using $ \xi_k = - c_k q_k^3 (1 - \pi_k q_k^{-2})^{-1} $ as above)
\begin{align*}
\Upsilon^*(E_i^a E_j E_i^b) &= \xi_i \Upsilon^*({}_i r(E_i^{a-1} E_j E_i^b)) \\
&= \Upsilon^*({}_i r(E_i^{a-1} E_j) E_i^b + \pi_i^{p(ai+j)} q^{(ai+j)\cdot i} E_i^{a-1} E_j {}_i r(E_i^b) ) \\
&= \Upsilon^*({}_i r(E_i^{a-1}) E_j E_i^b + \pi_i^{p(ai+j)} q^{(ai+j)\cdot i} E_i^{a-1} E_j {}_i r(E_i^b) )
\end{align*}  
and each of the two terms are of the form \eqref{sijeven}, and so the induction hypothesis applies; for $ a = 1 $ and not the base case we must have $b > 1 $ so we can use \ref{upr} on the other side. The base case here is $ \Upsilon^*(E_i E_j) = 0 $ for $ i \neq j $ which has been computed above.
\end{proof}

Note that $ \Upsilon $ is invertible in $\hat{\UU}$ and in fact $ \Upsilon^{-1} = \psi(\Upsilon) =: \overline{\Upsilon} $.
%cf. \cite[Corollary~2.13]{BW18a}.
\begin{cor}
 $	\overline{\Upsilon} \cdot \Upsilon = 1 $
\end{cor}
\begin{proof}
	Multiplying by $\Upsilon^{-1}$ on the left and right on both sides of \eqref{eq:intertwiner} gives us 
	\[
	\Upsilon^{-1}\psi_\imath(u) = \psi(u) \Upsilon^{-1}, \quad \text{for all } u \in \UUi
	\]
	Applying $ \psi $ to both sides and replacing $u$ with $ \ipsi(u) $, we have
	\[
	\psi{\Upsilon^{-1}}\psi(u) = \ipsi(u) \psi{\Upsilon^{-1}}, \quad \text{for all } u \in \UUi
	\]
	and so $\psi{\Upsilon^{-1}}$ also satisfies \eqref{eq:intertwiner} hence by uniqueness $\psi{\Upsilon^{-1}} = \Upsilon$ and so $ \Upsilon^{-1} = \ov{\Upsilon} $.
\end{proof}

%%%%%%%%%%%
%%%%%%%%%%%
\section{Integrality of actions of $\Upsilon$}
\label{sec:upintegrality}

As observed in the non-quantum covering case, it is neither expected nor required that the quasi-$K$-matrix for $\UUi$ beyond finite type is integral on its own cf. \cite{BW16}. For quantum symmetric pairs of super Kac-Moody type, the correct formulation is the integrality of the action of the quasi-$K$-matrix $\Upsilon$ i.e. we will see in this section that $\Upsilon$ preserves the integral $\A$-forms on integrable highest weight $\UU$-modules and their tensor products. 

\subsection{Definitions and background}
% see alsoSection~6 of \cite{BW18c}. 

We will use the following analogue of \cite[Lemma~2.2]{BW16}.

\begin{lem}
	\label{lem:2.2-BW16} Let $(M, B(M))$ be a based $\UU$-module and let $\la \in X$. Then, 
	\begin{enumerate}
		\item
		for $b \in B(M)$, the $\Q(q)$-linear map 
		$
		\pi_b: \UU^- {\bf 1}_{\ov{|b|+\la}} \longrightarrow M \otimes M(\la),  \;
		u \mapsto u(b \otimes \eta_\lambda)),
		$
		restricts to an $\A$-linear map
		$\pi_b: \aA \UU^- {\bf 1}_{\ov{|b|+\la}} \longrightarrow \aA M \otimes_\A \aM(\la)$;
		%,   \quad u \mapsto u(b \otimes 1).$
		
		\item
		we have $\sum_{b \in B(M)}  \pi_b (\aA \UU^- {\bf 1}_{\ov{|b|+\la}})= \aA M \otimes_\A \aM(\la)$.
	\end{enumerate}
\end{lem}

\begin{proof}
	The proof is the almost identical to the one for \cite[Lemma~2.2]{BW16}: the comultiplication has the same general formula as \cite[(2.1)]{BW16}, and the quantum covering analogue to (2.2) of \cite{BW16} can be found in \cite[(3.2)-(3.3)]{Cl14}.
\end{proof}

The quantum covering group $ \UUi $ also has a modified form $ \UUidot $ with idempotents via a familar construction cf. \cite[\S3.5]{BW18c}. The bar-involution $ \ipsi $ of $ \UUi $ then induces a bar-involution of the  $\Q^\pi$-algebra $ \UUidot $, also denoted $ \ipsi $, such that $ \ipsi(q) = \pi q^{-1} $ and $ \ipsi(B_i \one_\la) = B_i \one_\la $.
% include details for definitions for $ \UUidot $, and $ \ipsi $ on the $ \Q(q)^\pi $-algebra $\UUidot$.

\begin{definition}
	\label{def:mAUidot}
	Just as in 	Definition~3.10  of \emph{loc. cit.}, we define $_{\A} \UUidot$ to be the set of elements $u \in \UUidot$, such that $ u \cdot m \in {}_\A\UUdot$ for all $m \in {}_\A\UUdot$. Then $_{\A} \UUidot$  is clearly a $\A$-subalgebra of $\UUidot$ which contains all the idempotents $\one_\zeta$ $(\zeta \in X_\imath)$, and $_{\A} \UUidot = \bigoplus_{\zeta \in X_\imath}\,  {}_{\A} \UUidot \one_{\zeta}$.  
	
	Moreover, for $ u \in \UUidot $, we have $ u \in  {}_\A\UUdot $ if and only if $ u \cdot \one_\la \in {}_\A\UUdot $ for each $ \la \in X $ (cf. \cite[Lemma~3.20]{BW18b}).
\end{definition}

As a consequence of the existence of the $\imath^\pi$-divided powers, we have the following proposition.
\begin{prop}
	\label{prop:ipiDP}
	For any $i \in \I$ and $\mu \in X_\imath$, there exists an element $B^{(n)}_{i, \zeta} \in {}_{\A} \UUidot \one_{\zeta}$ satisfying the following 2 properties:
	\begin{enumerate}
		\item 	$\ipsi (B^{(n)}_{i, \zeta})=B^{(n)}_{i, \zeta}$;
		\item 	$ B^{(n)}_{i, \zeta} \one_{\lambda} = F^{(n)}_i \one_{\lambda} +\sum_{a < n}F^{(a)}_i {}_\A \UU^+\one_{\lambda} $, for $\one_{\lambda} \in {}_\A\UUidot$ with $\overline{\lambda} =\zeta$.
	\end{enumerate}
\end{prop}

The elements $ B^{(n)}_{i, \zeta} $ can be thought of as the `leading term' of the $\imath$-canonical basis elements in Proposition~\ref{prop:iCBUi} later.

\begin{definition}\label{def:AUi}
	Let  $\aAp \UUidot$ be the $\A$-subalgebra of ${}_\A\UUidot$ generated by the $\imath^\pi$-divided powers $B^{(n)}_{i, \zeta} \,(i \in \I)$
	for all $n \ge 1$ and $\zeta \in X_\imath$.
\end{definition}

Recall for $\lambda \in X$, we denote by $M(\lambda)$ the Verma module of highest weight $\lambda$ (see \cite[Section~2.6]{CHW13}). We denote the highest weight vector by $\eta_\lambda$. The following is an analogue of \cite[Lemma~6.3]{BW18c}.

\begin{lem}
	\label{lem:surjA1}Let $(M, B(M))$ be a based $\UU$-module. Let   $\la \in X$. Then, 
	\begin{enumerate}
		\item
		for $b \in B(M)$, the $\Q(q)$-linear map 
		$
		\pi_b: \UUidot {\bf 1}_{\ov{|b|+\la}} \longrightarrow M \otimes M(\la),  \;
		u \mapsto u(b \otimes \eta_\lambda)),
		$
		restricts to an $\A$-linear map
		$\pi_b: \aAp \UUidot {\bf 1}_{\ov{|b|+\la}} \longrightarrow \aA M \otimes_\A \aM(\la)$;
		%,   \quad u \mapsto u(b \otimes 1).$
		
		\item
		we have $\sum_{b \in B(M)}  \pi_b (\aAp \UUidot {\bf 1}_{\ov{|b|+\la}})= \aA M \otimes_\A \aM(\la)$.
	\end{enumerate}
\end{lem}

\begin{proof}
	Recall $ \aAp \UUidot \subset  {}_\A\UUidot$. Part (1) follows from Definition~\ref{def:mAUidot}.
	Part (2) is proven in the same way as \emph{loc. cit.} By part (1) we have  $\sum_{b \in B(M)}  \pi_b (\aAp \UUidot {\bf 1}_{\ov{|b|+\la}}) \subset \aA M \otimes_\A \aM(\la)$, and ${}_\A\UU^-$ has the increasing filtration 
	$$
	\A = {}_\A\UU^-_{\le 0} \subseteq {}_\A\UU^-_{\le 1} \subseteq \cdots \subseteq {}_\A\UU^-_{\le N} \subseteq \cdots
	$$
	where ${}_\A\UU^-_{\le N}$ is the $\A$-span of $\{F_{i_1}^{(a_1)}\ldots F_{i_n}^{(a_n)}   
	| a_1+\ldots + a_n\le N, i_1, \ldots, i_n \in I\}$, which induces an  increasing filtration $\{ \aM(\la)_{\le N} \}$ on $\aM(\la)$. 
	
	We can prove by induction on $N$ that ${}_\A M \otimes_\A \aM(\la)_{\le N} \subset \sum_{b \in B(M)}  \pi_b (\aAp \UUidot {\bf 1}_{\ov{|b|+\la}})$:
	
	Let $b \otimes \big(F_{i_1}^{(a_1)}\ldots F_{i_n}^{(a_n)} \eta_{\lambda} \big) \in {}_\A M \otimes_\A \aM(\la)_{\le N}$. 
	
	Now $ \Delta (B^{(a_1)}_{i_1, \zeta}) $ has the form  $ 1 \otimes F^{(a_1)}_{i_1, \zeta} + $ terms lower in filtration degree and so by Theorem~\ref{prop:ipiDP} and appropriate $\zeta \in X^\imath$ cf. \cite[Lemma 2.2]{BW16}, we have 
	\[
	B^{(a_1)}_{i_1, \zeta} \Big( b \otimes \big(F_{i_2}^{(a_2)}\ldots F_{i_n}^{(a_n)} \eta_{\lambda} \big) \Big) \in b \otimes \big(F_{i_1}^{(a_1)}\ldots F_{i_n}^{(a_n)} \eta_{\lambda} \big)  +  {}_\A M \otimes_\A \aM(\la)_{\le N-1}.
	\]
	The lemma follows.
\end{proof}

For $\la \in X^+$, we abuse the notation and denote also by $\eta_\la$ the image of $\eta_\lambda$ under the projection $p_\la: M(\la) \rightarrow L(\la)$. 
Note that $p_\la$ restricts to $p_\la: \aM(\la) \rightarrow \aL(\la)$. The next corollary follows from Lemma~\ref{lem:surjA1}. 

\begin{cor}
	\label{cor:surjA2}
	Let $\la \in X^+$, and let $(M, B(M))$ be a based $\UU$-module. Then, 
	\begin{enumerate}
		\item
		for $b \in B(M)$, the $\Q(q)$-linear map 
		$
		\pi_b: \UUi {\bf 1}_{\ov{|b|+\la}} \longrightarrow M \otimes L(\la),  
		\; u \mapsto u(b \otimes \eta_\la)$, restricts to an $\A$-linear map
		$\pi_b: \aAp \UUi {\bf 1}_{\ov{|b|+\la}} \longrightarrow \aA M \otimes_\A \aL(\la)$; 
		
		\item
		we have $\sum_{b \in B(M)}  \pi_b (\aAp \UUi {\bf 1}_{\ov{|b|+\la}})= \aA M \otimes_\A \aL(\la)$.
	\end{enumerate}
\end{cor}

%%%%%
\subsection{Integrality of actions of $\Upsilon$}

\subsubsection{}
Just as in \emph{loc. cit.}, the quasi-$K$-matrix $\Upsilon \in \widehat{\UU}^+$ induces a well-defined $\Q(q)$-linear map on $M \otimes L(\la)$: 
\begin{equation}
\label{eq:K-ML}
\Upsilon: M \otimes L(\la) \longrightarrow M \otimes L(\la),
\end{equation}
for any $\la \in X^+$ and any weight $\UU$-module $M$ whose weights are bounded above. %; cf. \cite[24.1.1]{Lu94}.

Recall \cite[\S5.1]{BW18b} that a $\UUi$-module $M$ equipped with an anti-linear involution $\ipsi$ 
is called {\em involutive} (or {\em $\imath$-involutive}) if
$$
\ipsi(u m) = \ipsi(u) \ipsi(m),\quad \forall u \in \UUi, m \in M.
$$

\begin{prop}
	\label{prop:compatibleBbar}
	Let $(M,B)$ be a based $\UU$-module whose weights are bounded above. We denote the bar involution on $M$ by $\psi$. Then $M$ is an $\imath$-involutive  
	$\UUi$-module with involution
	\begin{equation}
	\label{ibar}
	\ipsi := \Upsilon \circ \psi.
	\end{equation}
\end{prop}
%%%%%%%%%%%%%%%%%%%%%%%%%%%%%%%%%%%%%%%%%%%%%%
\begin{proof}
	Just as in \cite{BW18c}, since the weights of $M$ are bounded above, the action of $\Upsilon : M \rightarrow M$ is well defined. The rest of the argument is analogous to the one found in the proof of \cite[Proposition~5.1]{BW18b} (also \cite[Proposition~3.10]{BW18a}): using Theorem~\ref{thm:Upsilon}, we have
	\begin{equation*}
	\ipsi(um) = \Upsilon \psi(um) = \Upsilon \psi(u)\psi(m) = \ipsi(u) \Upsilon \psi(m) = \ipsi(u) \ipsi(m)
	\end{equation*}
	as required.
\end{proof}

\subsubsection{} 

Let $(M,B)$ be a based $\UU$-module whose weights are bounded above. Assume $\Upsilon: M\rightarrow M$ preserves the $\A$-submodule $\aM$. 
\begin{prop}
	\label{prop:quasi-KZ}
	The $\Q(q)$-linear map $\ipsi := \Upsilon \circ \psi$ preserves the $\A$-submodule $\aM \otimes_\A \aL(\la)$, for any $\la\in X^+$. 
\end{prop}

\begin{proof}
	The proof is again very similar: the $\UU$-module $M\otimes L(\la)$ is involutive with the involution $\psi:=\Theta \circ (\ov{\phantom{x}} \otimes \ov{\phantom{x}})$ where $\Theta$ is the quasi-$\mathcal R$-matrix from Proposition~\ref{prop:rmatrix}. It follows by an argument similar to \cite[Proposition~2.4]{BW16} that $\psi$ preserves the $\A$-submodule $\aM \otimes_\A \aL(\la)$: The statement is that for $ \lambda \in X^+ $ and $ (M, B(M)) $ be a based $ \UU $-module, the $ \Q(q) $-linear map
	$$ \Theta : M \otimes L(\lambda) \to M \otimes L(\lambda) $$
	preserves the $ \A $-submodule $ {}_\A M \otimes_\A {}_\A L(\lambda) $.
	
	We will write $ \overline{\phantom{B}} $ for $ \overline{\phantom{B}} \otimes \overline{\phantom{B}} $, which preserves the $ \A $-lattice $ {}_\A M \otimes_\A {}_\A L(\lambda) $. Thus, any $ x \in {}_\A M \otimes_\A {}_\A L(\lambda) $ can be recognized as $ x = \overline{x'} $ for some $ x' \in {}_\A M \otimes_\A {}_\A L(\lambda) $. By Lemma~\ref{lem:2.2-BW16}, $ x' = \sum_i \pi_{b_i}(u_i') $ (a finite sum), for some $ b_i \in B(M) $ and $ u_i' \in {}_\A \UU^- \one_{|b_i|+\lambda} $. Since $ {}_\A \UU^- \one_{|b_i|+\lambda} $ is preserved by the bar involution on $ \UUdot $, we have $ u_i' = \overline{u_i} $ for some $ u_i \in {}_\A \UU^- \one_{|b_i|+\lambda} $. Hence,
	$$
	x = \overline{x'} = \sum_i \overline{\overline{u_i}(b_i \otimes \eta_\lambda)}.
	$$	
	
	Using the property of the quasi-$\mathcal R$-matrix in Proposition~\ref{prop:rmatrix}, we have
	% Key property (see also \cite[Theorem~1]{Cl14})
	$$
	u \Theta(m \otimes m') = \Theta (\overline{\bar{u}(\bar{m} \otimes \bar{m'})}),
	$$
	for $ u \in \UUdot $, $ m \in M $ and $ m' \in L(\lambda) $. Taking $ m = b_i = \overline{b_i} $ and $ m' = \eta_\lambda = \overline{\eta_\la} $, this gives
	$$
	u(b_i \otimes \eta_\la) = \Theta (\overline{\bar{u}(\bar{b_i} \otimes \bar{\eta_\la})})
	$$
	since $ \Theta (\bar{b_i} \otimes \bar{\eta_\la}) = \bar{b_i} \otimes \bar{\eta_\la} $ (by construction, $ \Theta $ lies in a completion of $ \UU^- \otimes \UU^+ $), we have that
	$$
	\Theta(x) = \sum_i \Theta (\overline{\bar{u}(\bar{b_i} \otimes \bar{\eta_\la})}) = u_i (b_i \otimes \eta_\la) = \sum_i \pi_{b_i} (u_i),
	$$
	where the latter lies in $ {}_\A M \otimes_\A {}_\A L(\lambda) $ by Lemma~\ref{lem:2.2-BW16}, which completes the proof.
	
	Regarded as $\UUi$-module $M\otimes L(\la)$ is $\imath$-involutive with the involution $\ipsi:=\Upsilon \circ \psi$. 
	We can now prove that $\ipsi$ preserves the $\A$-submodule $\aM \otimes_\A \aL(\la)$.
	
	By  Corollary~\ref{cor:surjA2}(2), for any $x \in \aM \otimes_\A \aL(\la)$, we can write $x=\sum_k u_k (b_k\otimes \eta_\la)$, for $u_k \in \aAp \UUidot$ and $b_k \in B$. %; note we have replace $B$ by $B^\imath$ here thanks to Theorem~\ref{thm:iCBmodule}. 
	Since $M \otimes L(\la)$ is $\imath$-involutive, we have
	\begin{align}
	\label{eq:psipsi}
	\ipsi (x) =\sum_k \ipsi (u_k) \ipsi (b_k \otimes \eta_\la) 
	=\sum_k \ipsi (u_k) \Upsilon \psi (b_k \otimes \eta_\la)
	%=\sum_k \ipsi (u_k) \Upsilon  (b_k \otimes \eta_\la)
	=\sum_k \ipsi (u_k) (\Upsilon  b_k \otimes \eta_\la),
	\end{align}
	where we have used the fact that $\Delta (\Upsilon) \in \Upsilon \otimes 1 + \UU \otimes \UU^+_{>0}$ and $\psi(b_k \otimes \eta_\la) = \Theta(b_k \otimes \eta_\la) = b_k \otimes \eta_\la $ since $ \Theta $ is the sum of terms $ \Theta_\nu \in \UU_\nu^- \otimes \UU_\nu^+ $ and $ \Theta_0 = 1 \otimes 1 $. 
	By assumption we have $\Upsilon  b_k \in \aM$ and it follows by definition of $\aAp \UUidot$ that $\ipsi (u_k) \in  {}_\A\UUidot$. Applying Corollary~\ref{cor:surjA2}(2) again to \eqref{eq:psipsi}, we obtain that $\ipsi (x) \in \aM \otimes_\A \aL(\la)$. The proposition follows. 
\end{proof}

\begin{cor}
	\label{cor:Zform}
	The intertwiner $\Upsilon$ preserves the $\A$-submodule $\aM \otimes_\A \aL(\la)$. In particular, $\Upsilon$ preserves the $\A$-submodule $\aL(\la)$ of $L(\la)$.
\end{cor}
\begin{proof}
	Recall $\Upsilon =\ipsi \circ \psi$. The corollary follows from Proposition~\ref{prop:quasi-KZ} and the fact that $\psi$ preserves the $\A$-submodule $\aM \otimes_\A \aL(\la)$.
\end{proof}

\begin{cor} 
	\label{cor:tensor}
	Let $\lambda_i \in X^+$ for $1 \le i \le \ell$.  The involution $\ipsi$ on the $\imath$-involutive $\UUi$-module $L(\la_{1})\otimes \ldots \otimes L(\la_\ell)$ preserves the $\A$-submodule ${}_\A L(\la_{1})\otimes_{\A} \ldots \otimes_{\A} {}_\A L(\la_\ell)$.
\end{cor}

\begin{proof}
	The module $L(\la_{1})\otimes \ldots \otimes L(\la_\ell)$ is a based $\UU$-module whose weights are bounded above, and so the corollary follows by consecutive application of Proposition~\ref{prop:quasi-KZ}. 
\end{proof}

%	\blue{thefollowing theorem also follows, though since $ \osp(1|2n) $ is the only finite case here and we have already worked it out above, this is less significant in import than in \cite[Remark~6.10]{BW18c}.}

For finite type, we in fact have integrality of $ \Upsilon $ and not just its action.

\begin{thm}\label{thm:int}
	Assume $(\UU, \UUi)$ is of finite type. Write $ \Upsilon =\sum_{\mu} \Upsilon_\mu$. Then we have $\Upsilon_\mu \in \aA \UU^+$ for each $\mu$. 
\end{thm}
%%%%%%%%%%%%%%%%%%%%%%%%%%%%%%%%%%%%%%%%%%%%%%
\begin{proof}
	This follows by Corollary~\ref{cor:Zform} and applying $\Upsilon$ to the lowest weight vector $\xi_{-w_0\la} \in {}_\A L(\la)$, for $\la\gg 0$ (i.e., $\la \in X^+$ such that $\langle i, \la \rangle \gg 0$ for each $i$). 
\end{proof}

\section{$\imath$Canonical Basis on modules}
\label{sec:iCBmodules}

We call a $\UUi$-module $M$ a weight $\UUi$-module if $M$ admits a direct sum decomposition $M = \oplus_{\lambda \in X_\imath} M_{\lambda}$ such that, for any $\mu \in Y^\imath$, $\lambda \in X_\imath$, $m \in M_{\lambda}$, we have $K_{\mu} m = q^{\langle \mu, 
	\lambda \rangle} m$. 

We will make the following definition of based $\UUi$-modules (based on \cite[Definition~1]{BWW18}):

\begin{definition}\label{ad:def:1}
	Let $M$ be a weight $\UUi$-module over $\Q(q)^\pi$ with a given $\Q(q)^\pi$-basis $\B^\imath$. The pair $(M, \B^\imath)$ is called a based $\UUi$-module if the following conditions are satisfied:
	\begin{enumerate}
		\item 	$\B^\imath \cap M_{\nu}$ is a basis of $M_{\nu}$, for  any $\nu \in X_\imath$;
		\item 	The $\A$-submodule ${}_{\A}M$ generated by $\B^\imath$ is stable under ${}_{\A}\UUidot$;
		\item 	$M$ is $\imath$-involutive; that is, the $\Q^\pi$-linear involution $\ipsi: M \rightarrow M$ defined by 
		$\ipsi(q)= q^{-1}, \ipsi(b) = b$ for all $b \in \B^\imath$ 
		is compatible with the $\UUidot$-action, i.e., $\ipsi(um) = \ipsi(u) \ipsi(m)$, for all $u\in \UUidot, m\in M$;
		\item  
		Let $\mathbf{A}= \Q[[q^{-1}]]^\pi \cap \Q(q)^\pi$.	Let $L(M)$ be the $\mathbf{A}$-submodule of $M$ generated by $\B^\imath$. Then the image of $\B^\imath$ in $L(M)/ q^{-1}L(M)$ forms a $\Q^\pi$-basis in $L(M)/ q^{-1}L(M)$. % at $\infty$ for $M$.
	\end{enumerate}
\end{definition}	
We shall denote by $\mathcal L(M)$ the $\Z[q^{-1}]^\pi$-span of $\B^\imath$; then $\B^\imath$ forms a $\Z[q^{-1}]^\pi$-basis for $\mathcal L(M)$. We also define based $\UUi$-submodules and based quotient $\UUi$-modules in the obvious way.

By a standard argument using \cite[Lemma~9]{Cl14} (cf. \cite[Lemma~24.2.1]{Lu94}), we have the following generalization of \cite[Theorem~6.12]{BW18c} (cf. \cite[Theorem~5.7]{BW18b}): Recall that the partial order here is the one given by \eqref{eq:leq}, $ \la \leq \la' $ iff $ \la' - \la \in \N[I] $.

\begin{thm}
	\label{thm:piCBmodule}
	Let (M,B) be a based $\UU$-module whose weights are bounded above. Assume the involution $\psi_\imath$ of $M$ from Proposition~\ref{prop:compatibleBbar} preserves the $ \A^\pi $-submodule $ {}_\A M $.
	\begin{enumerate}
		\item The $ \UUi $-module $M$ admits a unique $\pi$-basis (called the $\imath$-canonical basis) $ B^\imath := \{ b^\imath | b \in B \} $, which is $ \ipsi $-invariant and of the form
		\begin{equation}
		b^\imath = b + \sum_{b' \in B, b' < b} t_{b;b'} b', \quad \text{for}\quad t_{b;b'} \in q^{-1} \Z^\pi[q^{-1}].
		\end{equation}
		
		\item $ B^\imath $ forms an $ \A^\pi $-basis for the $ \A^\pi $-lattice  $ {}_\A M $ (generated by $B$), and forms a $ \Z^\pi[q^{-1}] $-basis for the $ \Z^\pi[q^{-1}] $-lattice $ \mathcal{M} $ (generated by $B$).

		\item $( M, B^\imath )$ is a based $ \UUi $-module, where we call $ B^\imath $ the $ \imath $-canonical basis of $M$.
	\end{enumerate}
\end{thm}

Recall the based $\UU$-submodule $L(\la, \mu)$, for $\la, \mu \in X^+$, which in light of Theorem~\ref{thm:piCBmodule} can be viewed as a based $\UUi$-module. We denote this $\UUi$-module $ L^\imath(\la,\mu) $. A corollary of the theorem is the following cf. \cite[\S6]{BW18c}:

\begin{cor} \label{cor:basedtensor}
	Let $\lambda, \mu, \lambda_i \in X^+$ for $1 \le i \le \ell$, and $w\in W$. 
	\begin{enumerate}	
		\item 	$L(\lambda_1) \otimes \ldots \otimes L(\la_\ell)$ is a based $\UUi$-module, with the  $\imath$-canonical basis defined as Theorem~\ref{thm:piCBmodule}. 
		\item 	$L(w\lambda, \mu)$ is a based $\UUi$-submodule of $L(\lambda) \otimes L(\mu)$.
	\end{enumerate}
\end{cor}

\subsection{The element $\Theta^\imath$}
\label{subsec:Thetai}

Recall the quasi-$\mc R$ matrix $\Theta\in\widehat{\UU \otimes \UU}$ from \S\ref{subsec:bar+rmatrix} above. It follows from Theorem~\ref{thm:Upsilon} that $\Upsilon^{-1} \otimes \id$ and $\Delta(\Upsilon)$ are both in $\widehat{\UU \otimes \UU}$.

We define 
\begin{equation}\label{ad:eq:1}
\Theta^\imath = \Delta (\Upsilon) \cdot \Theta \cdot (\Upsilon^{-1} \otimes \id) \in \widehat{\UU \otimes \UU}.
\end{equation}

\begin{prop}[{cf. \upshape \cite[Proposition 3.2]{BW18a}}] \label{prop:quasiRB}
	For any $b\in \UUi$ one has
	\begin{align*}
	\Delta(\ipsi(b)) \cdot \Theta^\imath = \Theta^\imath \cdot (\ipsi\otimes \psi)\circ \Delta(b)
	\end{align*}
	in $\widehat{\UU \otimes \UU}$.
\end{prop}
%%%%%%%%%%%%%%%%%%%%%%%%%%%%%%%%%%%%%%%%%%%%%%
\begin{proof}
	Let $b\in \UUi$. Using the intertwiner relations one calculates
	\begin{align*}
	\Theta^\imath \cdot (\ipsi\otimes \psi)\circ \Delta(b) 
	&= \Delta(\Upsilon) \cdot \Theta \cdot (\Upsilon^{-1}\otimes 1)\cdot (\ipsi\otimes \psi)\circ \Delta(b)\\
	&= \Delta(\Upsilon) \cdot \Theta \cdot(\psi\otimes \psi)\circ \Delta(b)\cdot(\Upsilon^{-1}\otimes 1) \quad (\text{using Theorem~\ref{thm:Upsilon}})\\
	&=\Delta(\Upsilon)\cdot \Delta(\psi(b))\cdot \Theta \cdot(\Upsilon^{-1}\otimes 1) \quad (\text{using Prop~\ref{prop:rmatrix}})\\
	&=\Delta(\ipsi(b)) \cdot \Delta(\Upsilon)\cdot \Theta \cdot (\Upsilon^{-1}\otimes 1) \quad (\text{using Theorem~\ref{thm:Upsilon} again})
	\end{align*}
	which proves the proposition.
\end{proof}

We can write 
\begin{equation}
\label{eq:Thetamu}
\Theta^\imath = \sum_{\mu \in \N\I}\Theta^\imath_{\mu}, \qquad
\text{ where } \Theta^\imath_{\mu} \in  \UU \otimes \UU^+_\mu.
\end{equation}

\begin{lem}
	\label{Thetaparity}
	The first and second tensor factors of each term in $ \Theta^\imath_{\mu} \in  \UU \otimes \UU^+_\mu $ share the same parity.
\end{lem}
%%%%%%%%%%%%%%%%%%%%%%%%%%%%%%%%%%%%%%%%%%%%%%
\begin{proof}
	As we saw above, $ p(\Upsilon) = p(\Upsilon^{-1}) = 0 $ and so $ \Delta(\Upsilon) $ has the property that the first and second tensor factors of its terms share the same parity. By Proposition~\ref{prop:rmatrix}, $ \Theta_\nu $ also has this property, and so $ \Theta^\imath = \Delta (\Upsilon) \cdot \Theta \cdot (\Upsilon^{-1} \otimes \id) $ does as well.
\end{proof}

The following result is an analogue of \cite[Proposition~3.6]{Ko17}, which first appeared in \cite[Proposition~3.5]{BW18a} for the quantum symmetric pairs of (quasi-split) type AIII/AIV.
\begin{lem}\label{ad:lem:1}
	We have $\Theta^\imath_{\mu} \in \UUi \otimes \UU^+_\mu$, for all $\mu$. In particular,  we have $\Theta^\imath_{0} = 1\otimes 1$. 
\end{lem}
%%%%%%%%%%%%%%%%%%%%%%%%%%%%%%%%%%%%%%%%%%%%%%
\begin{proof}
	For any $i\in I$ one has
	\begin{align*}
	\Delta(B_i) = B_i \otimes K_i^{-1} + 1 \otimes F_i + c_i J_i \otimes E_i K_i^{-1}.
	\end{align*} Hence
	Proposition \ref{prop:quasiRB} implies that
	\begin{align*}
	\big(B_i \otimes K_i^{-1} + 1 \otimes F_i +  c_i J_i \otimes E_i K_i^{-1}\big)\cdot \Theta^\imath = \Theta^\imath \cdot
	\big(B_i \otimes J_i K_i + 1 \otimes F_i + \overline{c_i} J_i \otimes J_i K_i E_i \big).
	\end{align*}
	
	Rearranging this we obtain
	\begin{align}
	\label{BiTheta}
	\Theta^\imath (1 \otimes F_i) - (1 \otimes F_i) \Theta^\imath = (B_i \otimes K_i^{-1} + c_i J_i \otimes E_i K_i^{-1}) \Theta^\imath - \Theta^\imath (B_i \otimes J_i K_i + \overline{c_i} J_i \otimes J_i K_i E_i )
	\end{align}
	
	In each level $ \mu $, the left hand side is the sum of terms of the form
	\begin{align*}
	\big((\Theta^\imath_{\mu})_1 & \otimes (\Theta^\imath_{\mu})_2 \big) (1 \otimes F_i) - (1 \otimes F_i) \big((\Theta^\imath_{\mu})_1 \otimes (\Theta^\imath_{\mu})_2 \big) \\
	&= (\Theta^\imath_{\mu})_1 \otimes (\Theta^\imath_{\mu})_2 F_i - \blue{\pi_i^{p_1}} (\Theta^\imath_{\mu})_1 \otimes F_i (\Theta^\imath_{\mu})_2 \quad\text{where $p_k := p((\Theta^\imath_{\mu})_k)$, } k = 1,2 \\
	&= (\Theta^\imath_{\mu})_1 \otimes [(\Theta^\imath_{\mu})_2, F_i], \quad\text{since $\pi_i^{p_1} = \pi_i^{p_2}$ by Lemma~\ref{Thetaparity}} \\
	&= (\Theta^\imath_{\mu})_1 \otimes \bigg( { r_i((\Theta^\imath_{\mu})_2) J_i K_i - K_{-i} \pi_i^{p_2 - p(i)} {}_i r((\Theta^\imath_{\mu})_2) \over \pi_i q_i - q_i^{-1} } \bigg) \quad\text{by Proposition~\ref{prop:EF222}}
	\end{align*}
	
	Comparing this to terms on the right hand side of \eqref{BiTheta} with a factor of $ 1 \otimes J_i K_i $, we see that
	
	\begin{align}
	(1\otimes r_i) (\Theta^\imath_{\mu}) &= - (\pi_i q_i-q_i^{-1}) \Theta^\imath (B_i \otimes 1 + \overline{c_i} q_i^2 J_i \otimes E_i ) 
	\label{eq:irRt}
	\end{align}
	
	Then, the same induction as in \cite[Proposition~3.6]{Ko17} completes the proof, this time using Lemma~\ref{lem:rijr} as the appropriate analogue in the quantum covering group setting. 	
\end{proof}

The following is an analogue of \cite[Lemma~3]{BWW18}, used in the proof of a subsequent Theorem:
\begin{lem}\label{lem:intThetai}
	We have $ \Theta^\imath_{\mu} \in {}_\A \UU \otimes_\A {}_\A \UU^+_\mu $ for all $ \mu $.
\end{lem}

\begin{proof}
	The argument is analogous, using integrality of $ \Theta $ by Proposition~\ref{prop:rmatrix}, together with Theorem~ \ref{thm:int} in the definition of $ \Theta^\imath $.
\end{proof}

\begin{thm}
	Let $M$ be a based $\UUi$-module, and $\lambda \in X^+$. Then $\ipsi \stackrel{\rm def}{=} \Theta^\imath \circ (\ipsi \otimes \psi)$ is an anti-linear involution on $M \otimes L(\lambda)$, and $M \otimes L(\lambda)$ is  a based $\UUi$-module with a bar involution $\ipsi$.
\end{thm}
%%%%%%%%%%%%%%%%%%%%%%%%%%%%%%%%%%%%%%%%%%%%%%
\begin{proof}
	The anti-linear operator $\ipsi = \Theta^\imath \circ (\ipsi \otimes \psi): M \otimes L(\lambda) \rightarrow M \otimes L(\lambda)$ is well defined thanks to Lemma~\ref{ad:lem:1} and the fact that the weights of $L(\lambda)$ are bounded above. Then entirely similar to \cite[Proposition~3.13]{BW18a}, we see that $\ipsi^2=1$ and $M \otimes L(\lambda)$ is $\imath$-involutive in the sense of Definition~\ref{ad:def:1}(3). 
	
	The proof that $\ipsi$ preserves the $\A$-submodule ${}_\A M \otimes_{\A} {}_\A L(\lambda)$ is the same as the proof of Proposition~\ref{prop:quasi-KZ}. By assumption, $(M, \B^\imath(M))$ is a based $\UUi$-module. For any $b \in \B^\imath(M)$, define 
	\[
	\pi_b: {}_\A \UUidot \rightarrow {}_\A M \otimes_{\A} {}_\A L(\lambda), \quad u \mapsto u (b \otimes \eta_{\la}).
	\]
	Then, $\pi_b$ is well defined, since by Definition~\ref{def:mAUidot} and the following remark the coproduct preserves the integral forms, that is, $\Delta(u) (\one_{\mu} \otimes \one_{\nu})$ preserves ${}_\A M \otimes_{\A} {}_\A L(\lambda)$, for any $\mu \in X^\imath$ and $\nu \in X$. 
	
	Note that $ \ipsi (b \otimes \eta_{\lambda}) =  b \otimes \eta_{\lambda} $ for any $b \in \B^\imath(M)$. Following the proof of Lemma~\ref{lem:surjA1}, we have $\sum_{b \in \B^\imath(M)} \pi_b( {}_\A '\UUidot ) =  {}_\A M \otimes_{\A} {}_\A L(\lambda)$. Hence we also have $\sum_{b \in \B^\imath(M)} \pi_b( {}_\A \UUidot ) =  {}_\A M \otimes_{\A} {}_\A L(\lambda)$, since ${}_\A '\UUidot \subset {}_\A \UUidot$. By the same argument as before, we may conclude that $\ipsi$ preserves the $\A$-submodule ${}_\A M \otimes_{\A} {}_\A L(\lambda)$.
	
	We write  $\B = \{b^- \eta_{\la} \vert b \in \B(\lambda)\}$ for the canonical basis of $L(\lambda)$. Following the same argument as for \cite[Theorem~4]{BWW18} i.e. using Lemma~\ref{ad:lem:1} and Lemma~\ref{lem:intThetai} and \cite[Lemma~9]{Cl14}, we conclude that: 
	\begin{enumerate}
		\item
		for $b_1 \in \B^\imath, b_2\in \B$, there exists a unique element $b_1\diamondsuit_\imath b_2$ which is $\ipsi$-invariant such that $b_1\diamondsuit_\imath b_2\in b_1\otimes b_2 +q^{-1}\Z^\pi[q^{-1}] \B^\imath \otimes \B$;
		\item
		we have $b_1\diamondsuit_\imath b_2 \in b_1\otimes b_2 +\sum\limits_{(b'_1,b'_2) \in \B^\imath \times \B, |b_2'| < |b_2|} q^{-1}\Z^\pi[q^{-1}] \, b_1' \otimes b_2'$;
		\item
		$\B^\imath \diamondsuit_\imath \B :=\{b_1\diamondsuit_\imath b_2 \mid b_1 \in \B^\imath, b_2\in \B \}$ forms a $\Q(q)^\pi$-basis for  $M \otimes L(\lambda)$, an $\A^\pi$-basis for ${}_\A M  \otimes_{\A^\pi} {}_\A L(\lambda)$, and a $\Z^\pi[q^{-1}]$-basis for $\mathcal L(M)  \otimes_{\Z^\pi[q^{-1}]} \mathcal L(\lambda)$; 
		\item
		$(M\otimes L(\lambda), \B^\imath \diamondsuit_\imath \B)$ is a based $\UUi$-module.
	\end{enumerate}
\end{proof}

\section{Canonical basis on $\UUidot$}
\label{sec:CBmodified}
In this section, we formulate the main definition and theorems on canonical bases on the modified $\imath$quantum groups. The formulations are based on \cite[Section 7]{BW18c}, which in turn are generalizations of finite type counterparts in \cite[Section 6]{BW18b}.

%%%%
\subsection{The modified $\imath$quantum groups}

Recall the partial order $\leq$ on the weight lattice $X$ in \eqref{eq:leq}. The following proposition is a version of \cite[Proposition 7.1]{BW18c} in the quantum covering setting.

\begin{prop}\label{prop:props}
	Let $\lambda, \mu \in X^+$. %and let $\zeta = \wb \lambda +\mu$ and $\zeta_\imath = \overline{\zeta}$. 
	\begin{enumerate}
		\item	
		The $\imath$-canonical basis of the $ \UUi $-module $L^{\imath}(\lambda, \mu)$ is the basis
		$$\B^\imath(\la, \mu) =\big\{ (b_1 \diamondsuit_{\zeta_\imath} b_2 )_{\la,\mu}^{\imath} \vert 
		(b_1, b_2) \in \B^\imath \times \B \big\} \backslash \{0\},$$
		where 
		$(b_1 \diamondsuit_{\zeta_\imath} b_2 )_{\la,\mu}^{\imath}$ is $\ipsi$-invariant and lies in 
		\begin{equation*} % \label{eq:order}
		(b_1 \diamondsuit_{\zeta} b_2 ) (\eta_\lambda \otimes \eta_{\mu})    +\!\!\! \sum_{|b_1'|+|b_2'| \le |b_1|+|b_2|} \!\!\!q^{-1}\Z[q^{-1}] (b'_1 \diamondsuit_{\zeta} b'_2 ) (\eta_\lambda \otimes \eta_{\mu}). 
		\end{equation*}
		\item We have the projective system $\big \{ L^\imath(\lambda+\nu^\tau, \mu+\nu)  \big \}_{\nu \in X^+}$ of $\UUi$-modules, where 
		\begin{equation*} 
		\pi_{\nu+ \nu_1, \nu_1}: L^{\imath} (\lambda+\nu^\tau+\nu_1^\tau, \mu+\nu+\nu_1) \longrightarrow L^{\imath} (\lambda+\nu^\tau, \mu+\nu), \quad \nu, \nu_1 \in X^+,
		\end{equation*}
		is the unique homomorphism of $\UUi$-modules such that 
		\[\pi(\eta_{\lambda+\nu^\tau+\nu_1^\tau} \otimes \eta_{\mu+\nu+\nu_1}) = \eta_{\lambda+\nu^\tau} \otimes \eta_{\mu+\nu}.\] 
		
		\item The projective system in (2) is asymptotically based in the following sense: for fixed $(b_1, b_2) \in \B^\imath \times \B$ and any $\nu_1 \in X^+$, as long as $\nu \gg 0$, we have 
		\[ 
		\pi_{\nu+ \nu_1, \nu_1} \big((b_1 \diamondsuit_{\zeta_\imath} b_2 )_{\lambda+\nu^\tau+\nu_1^\tau,\mu+\nu+\nu_1}^{\imath}\big) =  \big((b_1 \diamondsuit_{\zeta_\imath} b_2 )_{\lambda+\nu^\tau,\mu+\nu}^{\imath}\big).
		\]
	\end{enumerate}
\end{prop}

\begin{proof}
	Claim (1) is just a reformulation of \ref{cor:basedtensor}. %Corollary~\ref{cor:iCBbased}. 
	Claim (2) follows by the same proof as \cite[Proposition~6.12]{BW18b}, using the quasi-$ \mathcal{R} $-matrix in Proposition~\ref{prop:rmatrix}.
	
	Claim (3) is the same as \cite[Proposition~6.16]{BW18b}, and we can do without the mild modification needed in \cite{BW18c} since the module $L(\nu^\tau+\nu)$ is finite dimensional. 
	%But this can be replaced by the fact that only finitely many $a(b', b'')$ are non-zero {\em loc cit}. The rest is exactly the same. 
	%
	%Note that the (fixed) parameters plays no essential role here as we are taking  $\nu \gg 0$.
\end{proof}

Proposition~\ref{prop:props} is the main mechanism of proof in the following version of \cite[Theorem 7.2]{BW18c} (see also \cite[Theorem 6.17]{BW18b}), granting the $\imath$-canonical basis for $ \UUidot $:

\begin{prop}  
	\label{prop:iCBUi}
	Let ${\zeta_\imath} \in X_{\imath}$ and $(b_1, b_2) \in B \times B$. 
	\begin{enumerate}
		\item	
		There is a unique element $u =b_1 \diamondsuit^\imath_{\zeta_\imath} b_2  \in \UUidot$ such that 
		\[
		u (\eta_\lambda \otimes \eta_\mu) 
		= (b_1 \diamondsuit_{\zeta_\imath} b_2)_{\lambda,\mu}^\imath				\in L^{\imath} (\lambda, \mu), 
		\]
		for all $\lambda, \mu \gg0$ with $\overline{\lambda+\mu} = {\zeta_\imath}$.
		\item	The element $b_1 \diamondsuit^\imath_{\zeta_\imath} b_2$ is $\ipsi$-invariant.
		\item	The set $\dot{\B}^\imath =\{ b_1 \diamondsuit^\imath_{\zeta_\imath} b_2 \big \vert {\zeta_\imath} \in X_\imath, (b_1, b_2) \in B_{\Iblack} \times B \}$ 
		forms a $\Q(q)^\pi$-basis of $\UUidot$ and an $\A^\pi$-basis of ${}_\A\UUidot$.
	\end{enumerate}
\end{prop}

%%%%%%%%%%%%%%%%%%%%%%%%%%%%%%%%%%%%%%%%%%%%%%%%%%%%%%%%%%%%%%%%%%%%%

\end{document}